\documentclass[11pt, a4paper, leqno]{amsart}

\usepackage[left=5in, right=4in, top=3in, bottom=3in, includefoot, headheight=13.6pt]{geometry}

\usepackage{amsmath}
\usepackage{amsfonts}
\usepackage{amssymb}
\usepackage{amsthm}
\usepackage{geometry}
\usepackage{latexsym}
\usepackage{mathrsfs}
\usepackage[all]{xy}

\theoremstyle{plain}
\newtheorem{theorem}{Theorem}[section]

\newtheorem{lemma}[theorem]{Lemma}
\newtheorem{corollary}[theorem]{Corollary}
\newtheorem{proposition}[theorem]{Proposition}

\theoremstyle{definition}
\newtheorem{example}[theorem]{Example}
\newtheorem{question}[theorem]{Question}
\newtheorem{definition}[theorem]{Definition}
\newtheorem{defn/propn}[theorem]{Definition/Proposition}
\newtheorem{remark}[theorem]{Remark}

\newtheorem{innercustomthm}{Theorem}
\newenvironment{customthm}[1]
  {\renewcommand\theinnercustomthm{#1}\innercustomthm}
  {\endinnercustomthm}

\newcommand{\C} {\mathbb C}
\newcommand{\N}{\mathbb N}

\newcommand{\R}{\mathbb R}
\newcommand{\T}{\mathbb T}
\newcommand{\Z}{\mathbb Z}

\newcommand{\Balg} {\mathcal{B}}
\newcommand{\Alg} {\mathcal{A}}

\newcommand{\ip}[2]{\langle #1,#2 \rangle}

\newcommand{\im}{\textup{im}}
\newcommand{\Hil}{\mathcal{H}}
\newcommand{\Dirac}{\mathcal{D}}
\newcommand{\semidir}{{\rtimes}}

\newcommand{\subs}{\subseteq}

\title{Spectral Metric Spaces on Extensions of C*-Algebras}
\author{Andrew Hawkins}
\address{Kendal College,   Milnthorpe Road, Kendal, Cumbria LA9 5AY, England}
\email{ah3@Kendal.ac.uk }

\author{Joachim Zacharias}
\address{School of Mathematics and Statistics, University of Glasgow, 15 University Gardens, Glasgow, G12 8QW,
Scotland}
\email{Joachim.Zacharias@glasgow.ac.uk  }

\thanks{\emph{Supported by:}  EPSRC Grant EP/I019227/1-2}


\begin{document} \maketitle

\begin{abstract} \noindent We construct spectral triples on C*-algebraic extensions of unital C*-algebras by stable ideals satisfying a certain Toeplitz type property using given spectral triples on the quotient and ideal. Our
construction behaves well with respect to summability and produces new spectral quantum metric spaces out of given ones. Using our construction we find new spectral triples on the quantum 2- and  3-spheres giving a new perspective on these algebras in noncommutative geometry.
\end{abstract}

\section{Introduction}\label{intro}

\subsection{Background.}

Spectral triples, a central concept  of noncommutative geometry, provide an analytical language for geometric objects. A prototype is given by the triple $(C^1(\mathcal{M}), L^2(\mathcal{M},\mathcal{S}), \Dirac)$ which is a spectral triple on the algebra $C(\mathcal{M})$ of continuous functions on $\mathcal{M}$, where $\mathcal{M}$ is  a compact Riemannian manifold equipped with a spin$^C$ (or spin) structure, $C^1(\mathcal{M})$ a dense ``smooth" subalgebra of  $C(\mathcal{M})$ and $\Dirac$ is the corresponding Dirac operator acting on $L^2(\mathcal{M},\mathcal{S})$.

Connes \cite{Con1}, \cite{Con2} introduced spectral triples as a potential means of describing the homology and index theoretic aspects in the more general language of (locally) compact topological spaces, as well as to develop a theory of cyclic cohomology mimicking the de-Rham cohomology theory of manifolds. Further, Connes shows that geometric information about a Riemannian manifold $\mathcal{M}$, such as the geodesic distance and dimension, can all be recovered from the Dirac triple on $C(\mathcal{M})$.

Spectral triples are motivated by Kasparov theory and can be regarded as  ``Dirac-type" or elliptic operators on general C$^*$-algebras (usually assumed separable). In particular a spectral triple defines a $K$-homology class. Spectral triples with good properties can be used to encode geometric information on a C$^*$-algebra. Besides the link between summability and dimension which is well understood in the commutative case, we mention two examples of current areas of research.

The first is  Connes' reconstruction programme, the aim of which is to find conditions or axioms under which a spectral triple on a  commutative C$^*$-algebra can provide the spectrum of the algebra with the structure  of a manifold. Several reconstruction theorems have been suggested in what has become a very prominent area of research (see for example \cite{LRV}, \cite{Con3}). Besides the noncommutative tori, there do not seem to be many examples of noncommutative C$^*$-algebras at present for which this sort of analysis can be extended to.

The second one is the idea to regard spectral triples  as noncommutative (quantum) metric spaces, beginning with Connes' observation \cite{Con2} that the Dirac triple on a Riemannian spin$^C$ manifold $\mathcal{M}$ recovers the geodesic distance between two points on the manifold. In fact Connes' expression  for the geodesic distance extends immediately to a metric on the space of probability measures on $\mathcal{M}$. In more recent and general language, a spectral triple on a C$^*$-algebra determines a \textit{Lipschitz seminorm} on the self-adjoint part of the smooth subalgebra, an analogue of the classical notion of Lipschitz continuous functions. In e.g. \cite{Ri1}, \cite{Ri2} and \cite{Ri3} Rieffel studies Lipschitz seminorms of this kind extensively. Under mild conditions such a seminorm defines a metric on the state space of the algebra by a formula analogous to the manifold case.  However, in general, Lipschitz seminorms and corresponding metrics may be quite arbitrary.  A  natural condition one would expect this metric to satisfy is that it induces the weak-$*$-topology on the state space and Rieffel makes this the defining condition of his notion of a \textit{quantum metric space}. Rieffel found a very useful characterisation of this metric condition for unital C$^*$-algebras  (\cite{Ri1}, cf. Prop.\ref{metric condition} below for the statement). We will refer to this condition as Rieffel's metric condition.  Latr{\'e}moli{\`e}re later extended much of this work to non-unital C$^*$-algebras in \cite{Lat1} and \cite{Lat2}. A C$^*$-algebra equipped with a spectral triple satisfying this metric condition is sometimes called a \textit{spectral metric space}.

Despite the longevity of spectral triples as a subject of study, general methods of constructing spectral triples on C$^*$-algebras are not well understood, much less still those satisfying the metric condition. There have been successful constructions of so-called spectral metric
spaces on certain noncommutative C*-algebras, such as approximately finite dimensional algebras (\cite{CI2}), group C*-algebras of discrete hyperbolic groups (\cite{OR}) and algebras arising as $q$-deformations of the function algebras of simply connected simple compact Lie groups (\cite{NT}).

Building on previous authors' works, we are particularly interested in `building block' constructions i.e. constructing new spectral triples from old ones, which is also in the sprit of permanence properties. This point of view has been used by various authors to attempt to construct spectral triples on crossed products of C*-algebras by certain discrete groups (\cite{BMR}, \cite{HSWZ}). More specifically, the authors of those two references study C$^*$-dynamical systems $(A,G,\alpha)$ in which the algebra $A$ is equipped with the structure of a spectral triple with good metric properties and consider under what conditions it is possible to write down a spectral triple on $A \semidir_{r, \alpha} G$ using a natural implementation of the external product in Kasparov theory. It turns out that a necessary and sufficient condition is the requirement that the action of $G$ essentially implements an isometric action on the underlying spectral metric space. This is satisfied for a variety of group actions and, via this construction, the authors in collaboration with A. Skalski and S. White (\cite{HSWZ}) were able to write down spectral triples with good metric properties on both the irrational rotation algebras and the Bunce-Deddens algebras and some of their generalisations.
     
Spectral triples define Baaj-Julg cycles, the unbounded analogue of a Kasparov bimodule in KK-theory (\cite{BJ}). This perspective is increasingly being examined by various authors to write down spectral triples on C$^*$-algebras by means of an unbounded version of Kasparov's internal product, which is defined for C$^*$-algebras $A$, $D$, $B$ and $p, q \in \{0,1\}$ as a map $\otimes_B: KK^p(A,D) \times KK^q(D,B) \to KK^{p+q}(A,B)$. There are a couple of important recent developments in this area: Gabriel and Grensing (\cite{GG}) consider the possibility of writing down spectral triples on certain Cuntz-Pimsner  algebras, generalising the setting of ordinary crossed products by $\Z$ but with the same property that the triples they construct represent the image of a given triple under the boundary map in the resulting six-term exact sequence. They succeed in implementing these techniques to construct a variety of spectral triples on certain quantum Heisenberg manifolds. Goffeng and Mesland (\cite{GM}) investigate how the Kasparov product can, under suggested modifications, be used to write down spectral triples on Cuntz-Krieger algebras, beginning with the spectral triple on the underlying subshift space. It is anticipated therefore that there will be a considerable interest in the interplay between spectral triples and the Kasparov product in the near future.

In this paper we construct spectral triples on extensions of C*-algebras out of given ones on the ideal and the quotient algebra. We are, however, primarily concerned with those which satisfy Rieffel's metric condition, thus implementing the structure of a quantum metric space on the extension, beginning with related structures on both the quotient and ideal. Techniques in Kasparov theory will be important to us too, but certain technical difficulties will prevent us from being able to give a full description of the resulting triples in terms of their representatives in K-homology. We remark that the ideas in this paper are closely linked to those of Christensen and Ivan (\cite{CI1}) and are to some extent a generalisation of their results.

\subsection{Outline of the paper.} \label{outline}
We assume throughout the paper that all C$^*$-algebras and Hilbert spaces are separable.
Given a C$^*$-algebra $E$ and an essential ideal $I \subs E$, and given spectral triples on both $I$ and $E/I$, is there any way of constructing a spectral triple on $E$ out of the given spectral triples? In this paper we will be looking at the situation in which the quotient is a unital C$^*$-algebra $A$ and the ideal is the tensor product of a unital C$^*$-algebra by the algebra of compact operators, that is, we consider extensions of the form,
\begin{eqnarray}
\xymatrix{ 0 \ar[r] & \mathcal{K} \otimes B \ar[r]^{\iota} & E \ar[r]^{\sigma} & A \ar[r] & 0}.
\end{eqnarray}
This is a generalisation of the situation considered by Christensen and Ivan \cite{CI1}, who looked at short exact sequences of
the form,
\begin{eqnarray}
\xymatrix{ 0 \ar[r] & \mathcal{K} \ar[r]^{\iota} & E \ar[r]^{\sigma} & A \ar[r] & 0}.
\label{ext-k}
\end{eqnarray}

They exploited the fact that a certain class of C$^*$-extensions by compacts (those which are semisplit) can be spatially represented over a Hilbert space: as outlined in Section 2.7 of \cite{HR}, we can regard $E$ as a subalgebra of the bounded operators on an infinite dimensional Hilbert space $H$ generated by compacts on $PH$ and the operators $\{ P\pi_A(a)P \in B(H): a \in A \}$, where $\pi_A: A \to B(H)$ is a faithful representation and $P \in B(H)$ is an orthogonal projection with infinite dimensional range. The algebra acts degenerately, only on the subspace $PH$.

There is a certain generalisation of this picture for semisplit extensions by general stable ideals of the form (\ref{extension}) which is due to Kasparov (\cite{Kas3}, see also \cite{Bla}). For such extensions, $E$ can always be regarded as a subalgebra of $\mathcal{L}_B(\ell_2(B)) = \mathcal{L}_B$, the C$^*$-algebra of bounded $B$-linear and adjointable operators on the Hilbert module $\ell_2(B)$. In fact, using semisplitness, there is a representation $\pi : A \to \mathcal{L}_B(\ell_2(B) \oplus \ell_2(B)) \cong \mathcal{L}_B(\ell_2(B)) $ and a projection $P \in \mathcal{L}_B$ such that  $E$ is generated by $P\pi (A) P $ and $P (\mathcal{K} \otimes B) P= P (\mathcal{K}_B)P$ (cf. Section \ref{ext} for more details). However, to construct spectral triples on $E$ we need a representation on a Hilbert space, not  Hilbert module. Our given spectral triples come with concrete representations $\pi_A : A \to B(H_A)$ and $\pi_B : B \to B(H_B)$ on Hilbert spaces. It seems reasonable to study those extensions which act naturally on the tensor product $H_A \otimes H_B$, possibly degenerately i.e. only on a subspace of this tensor product. More precisely, we consider representations of the form
$$
\pi: E \to B(H_A \otimes H_B), \;\;\; \;\; \pi(\mathcal{K} \otimes B) = \mathcal{K}(H_0)\otimes \pi_B(B)
$$
where $H_0$ is an infinite dimensional subspace of $H$, and $\pi (k \otimes b) = \phi (k) \otimes \pi_B(b)$ with $ \phi : \mathcal{K} \to \mathcal{K} (H_0)$ an isomorphism. Not all extensions can be brought into this form. In Section \ref{ext} we show that it is possible if the Busby invariant satisfies a certain factorisation property. In order to describe it in somewhat more detail recall that a short exact sequence of C$^*$-algebras,
\begin{eqnarray*}
\xymatrix{ 0 \ar[r] & \mathcal{K} \otimes B \ar[r]^{\iota} & E \ar[r]^{\sigma} & A \ar[r] & 0}
\end{eqnarray*}
is characterised by a $^*$-homomorphism $\psi: A \to \mathcal{Q}_B$, the \textit{Busby invariant}, where
$\mathcal{Q}_B := \mathcal{L}_B / \mathcal{K}_B$ is sometimes called the generalised Calkin algebra with respect to the C$^*$-algebra $B$. Since $\mathcal{L}_B \cong \mathcal{M} (\mathcal{K} \otimes B)$ there is an embedding of the ordinary Calkin algebra $\mathcal{Q}= \mathcal{M}(\mathcal{K})/\mathcal{K}$ into $\mathcal{Q}_B$ and the condition which characterises the extensions we consider is that there exists a semisplit extension of $A$ by $\mathcal{K}$ of the type (\ref{ext-k}) with Busby invariant $\psi_0: A \to \mathcal{Q}$ such that $\psi$ factors through $\psi_0$ and the natural inclusion of $\mathcal{Q}$ into  $\mathcal{Q}_B= \mathcal{M}(\mathcal{K} \otimes B)/\mathcal{K\otimes B}$. In KK-theoretic language, we need the class of $\psi$  in $KK^1(A,B)= \text{Ext}(A,B)^{-1}$ to factor into the class of an extension
$\psi_0$ in $K^1(A)=\text{Ext}^{-1}(A)$ and the $K_0(B)$-class of $1_B \in B$, i.e.
$$
[\psi] = [\psi_0] \otimes [1_B].
$$
In this situation we can view the algebra $E$ as a concrete subalgebra of $B(H_A \otimes H_B)$ generated by
elements of the form $PaP \otimes 1_B$ and $k \otimes b$, where $P \in B(H_A)$ is an orthogonal infinite dimensional projection, $a \in A$, $b \in B$ and $k$ is a compact operator in $PH$ (cf. Corollary \ref{tensorial}). Throughout the paper we will assume that our extension is of this form. To avoid technicalities we will  assume that $PaP \cap \mathcal{K} = \{0\}$ which is true for essential extensions and can always be arranged by replacing $\pi_A$ by an infinite ampliation $\pi_A^{\infty}$.

Starting from a pair of spectral triples $(\Alg, H_A, \Dirac_A)$ on $A$ and $(\Balg, H_B, \Dirac_B)$ on $B$ (cf. Section \ref{spectriples} for notation), Kasparov theory can be used to write down a spectral triple on $A \otimes B$ whose representative in K-homology is the \textit{external Kasparov product} of the representatives of $(\Alg, H_A, \Dirac_A)$ and $(\Balg, H_B, \Dirac_B)$. When the spectral triple on $A$ is odd and the spectral triple on $B$ is even, i.e. there is a direct sum representation $\pi_B^+ \oplus \pi_B^-$ and
$\Dirac_B$, acting on $H_B \otimes \C^2$, decomposes as the matrix,
$$
\begin{bmatrix} 0 & \Dirac_B^+ \\ \Dirac_B^- & 0 \end{bmatrix},
$$
then the spectral triple can be defined on the spatial tensor product $A \otimes B$ acting on the Hilbert space $H_A
\otimes H_B \otimes \C^2$ via the representation $(\pi_A \otimes \pi_B^+) \oplus (\pi_A \otimes \pi_B^-)$ with the
Dirac operator,
$$
\begin{bmatrix} \Dirac_A \otimes 1 & 1 \otimes \Dirac_B^+ \\ 1 \otimes \Dirac_B^- & -\Dirac_A \otimes 1
\end{bmatrix},
$$
(which can be interpreted as sum two graded tensor products), whereas the product of two ungraded triples is represented by the matrix
$$
\begin{bmatrix} 0 & \Dirac_A \otimes 1 - i \otimes \Dirac_B\\ 
 \Dirac_A \otimes 1 + i \otimes \Dirac_B & 0
\end{bmatrix},
$$
acting on $H_A \otimes \C^2$ (see for example \cite{Con1} pp. 433-434).
We mention these formulae since they serve as an inspiration for the Dirac operator we are going to write down for the extension. In fact our operator will be a combination of these two formulae which makes it difficult to interpret our construction in K-homological terms. 

%On the other hand, it is really the \textit{internal Kasparov product}, not the external one, which %is applicable to our methods. We will show that there is a canonical way of producing a cycle, %$\sigma^* \in KK^0(E,A)$, and a cycle $\tau^* \in KK^1(E,B)$, beginning from any such short %exact sequence of C$^*$-algebras. Motivated by the description of K-homology for algebras %arising as split extensions, the theory is to seek spectral triples on $E$ which reproduce the K-homology class of $E$ given by the direct sum,
%$$
%(\sigma^* \otimes_A [(\Alg, H_A, \Dirac_A)]) \oplus (\tau^* \otimes_B [(\Balg, H_B, \Dirac_B)]) %\in KK(E,\C).
%$$
%Recent work by Kaad \cite{Kaa}, Mesland \cite{Mes}, Kaad-Lesch \cite{KL} and others have %extended the works of Baaj-Julg (\cite{BJ}) by showing that it is possible, under suitable %conditions, to represent the internal Kasparov product of unbounded KK-cycles using explicit %formulas. The ideas are extremely delicate and require a lot of sophisticated techniques, which %we will not endeavour to address in this paper.

Returning to our set-up, our assumptions imply that we can, omitting representations, write down an isomorphism,
\begin{eqnarray*}
E \cong \mathcal{K}(P H_A) \otimes B + PAP \otimes \C I_B,
\end{eqnarray*}
where $[P, a]$ is a compact operator on $H_A$.  
$E$ can be regarded as a concrete subalgebra of $B(H_A \otimes H_B)$ acting degenerately (effectively only on $PH_A \otimes H_B$). The corresponding representation is denoted by $\pi$.  There is another representation  $\pi_{\sigma}: E \to B(H_A \otimes H_B)$ given as the composition of the quotient map $\sigma: E \to A$ and the natural representation $\pi_A \otimes 1$ of $A$ on $B(H_A \otimes H_B)$. $\pi_{\sigma}$ is non-degenerate but not faithful. The information coming from both representations is essential to writing down a
spectral triple on $E$ which encodes the metric behaviour of both the ideal and quotient parts of the extension.
%In the above, we call the triple $(\pi_A, \pi_B, P)$ a \textit{Toeplitz triple} for the extension.  
We will use this information, the presence of Dirac operators $\Dirac_A$ on $H_A$ and $\Dirac_B$ on $H_B$ together with the
aforementioned ideas in Kasparov theory to build spectral triples on $E$. The representation of this triple will be a suitable combination of the two representations of $E$.
       
In order to build a spectral triple we need the further requirement  that $P$ commutes with $\Dirac_A$
which then decomposes into the direct sum $\Dirac_A = \Dirac_A^p \oplus \Dirac_A^q$, where $\Dirac_A^p = P
\Dirac_A P$, $\Dirac_A^q = Q \Dirac_A Q$ and $Q = 1-P$. Next we require
\begin{eqnarray*}
[P, \pi_A(a)] \in \mathcal{C}(H_A) \;\;\; \forall a \in \Alg,
\end{eqnarray*}
where $\mathcal{C}(H_A)$ is the dense $^*$-subalgebra of elements $x \in \mathcal{K}(H_A)$ such that $x \textup{dom}(\Dirac_A) \subseteq \textup{dom}(\Dirac_A)$ and both $x \Dirac_A$ and $\Dirac_A x$ extend to bounded operators ($P$-regularity cf. Defs.\ref{D-diff}, \ref{Toeplitz-type}). One may think of  $\mathcal{C}(H_A)$ as the dense $^*$-subalgebra of `differentiable compacts', hence the notation used. 

We summarise our assumptions on the extension and spectral triples in the following definition which is consistent with the article \cite{CI1}.
     
\begin{definition}\label{Toeplitz type}
Let $\pi_A : A \to B(H_A)$ and $\pi_B : B \to B(H_B)$ be faithful representations, where $A, B$ are separable unital  C$^*$-algebras and $H_A ,H_B$ separable Hilbert spaces. 
The extension 
\begin{eqnarray}
\xymatrix{ 0 \ar[r] & \mathcal{K} \otimes B \ar[r]^{\iota} & E \ar[r]^{\sigma} & A \ar[r] & 0}
\label{extension}
\end{eqnarray}
is said to be of \textit{Toeplitz type} if there exists an infinite dimensional  projection $P \in B(H_A)$ such that $$[P,\pi_A(a)] \in \mathcal{K}(H_A),$$  $$E \cong   \mathcal{K}(P H_A) \otimes \pi_B(B) + P\pi_A(A)P \otimes \C I_B$$ and $$ \mathcal{K}(P H_A) \otimes \pi_B(B) \cap  P\pi_A(A)P \otimes \C I_B= \{0\}.$$ $(\pi_A, \pi_B , P)$ is then referred to as a \textit{Toeplitz triple} for the extension. 

If moreover $(\Alg, H_A ,\Dirac_A)$ is a spectral triple  such that $\Dirac_A$ and $P$ commute  and $[P, \pi_A(a)] \in \mathcal{C}(H_A)$ for all $a \in \Alg$, then the quadruple $(\Alg, H_A, \Dirac_A, P)$ (or just $P$)  is said to be  of \textit{Toeplitz type}.  

A Toeplitz type quadruple is said to be $P$-\textit{injective} if $\ker(\Dirac_A^p \cap PH_A) = \{0\}$. 
\end{definition}

When $P$ coincides with the orthogonal projection into the closed span of the positive eigenspace for $\Dirac_A$, then the smoothness assumption turns out to be equivalent to saying that not only $[\Dirac_A, \pi_A(a)]$ but also $[|\Dirac_A|, \pi_A(a)]$ is a bounded operator for each $a \in \Alg$, which is related to the concept of regularity for spectral triples in Riemannian geometry due to Connes and Moscovici (\cite{CMh2}).

We state here the two main results of our paper  asserting the existence of spectral triples with good summability properties on Toeplitz-type extensions under the assumption that $(\Alg, H_A, \Dirac_A, P)$ is of Toeplitz type 
 (Theorem \ref{ConstExt}) and that Rieffel's metric condition is preserved under the mild extra assumption that $\Dirac_A$ is $P$-injective (Theorem \ref{ConstExt2}). 

Before we can do so we need to introduce further notation. Let $\Pi_1, \Pi_2 : E \to B(H_A \otimes H_B \otimes \C^2)$ be the representations given by
\begin{eqnarray}
\Pi_1 = \pi_{\sigma} \oplus \pi_{\sigma} \textup{ and } \Pi_2 = \pi \oplus \pi_{\sigma}
\end{eqnarray}
and consider operators $\Dirac_{1}, \Dirac_{2}, \Dirac_{3}$ on $H_A \otimes H_B \otimes \C^2$ given by
\begin{eqnarray}
\Dirac_{1} = 
\begin{bmatrix} \Dirac_A \otimes 1  & 1 \otimes \Dirac_B  \\ 
1 \otimes \Dirac_B  & -\Dirac_A \otimes 1 
\end{bmatrix},
\end{eqnarray}
\begin{eqnarray}
\Dirac_{2} = 
\begin{bmatrix} \Dirac_A^q \otimes 1 & \Dirac_A^p \otimes 1 \\ 
\Dirac_A^p \otimes 1 & -\Dirac_A^q \otimes 1
\end{bmatrix}
= 
\begin{bmatrix} \Dirac_A^q  & \Dirac_A^p  \\ 
\Dirac_A^p  & -\Dirac_A^q 
\end{bmatrix}
\otimes I
\end{eqnarray}
and 
\begin{eqnarray}
\Dirac_{3} = 
\begin{bmatrix}  1 \otimes \Dirac_B & 0 \\ 
0 & 1 \otimes \Dirac_B
\end{bmatrix}
= 
I \otimes
\begin{bmatrix} \Dirac_B  & 0  \\ 
0  & \Dirac_B
\end{bmatrix}
\end{eqnarray}

\begin{customthm}{\ref{ConstExt}}
Let $A$ and $B$ be unital C$^*$-algebras and suppose that $E$ arises as the short exact sequence
(\ref{extension}) and that there exist spectral triples $(\Alg,H_A,\Dirac_A)$ on $A$ and
$(\Balg,H_B,\Dirac_B)$ on $B$, represented via $\pi_A$ and $\pi_B$ respectively, and an orthogonal projection $P \in B(H_A)$ such that $(\Alg, H_A, \Dirac_A, P)$ is of Toeplitz type. Let 
\begin{eqnarray*}
\Pi = \Pi_1\oplus \Pi_2 \oplus \Pi_{2}, \;\;H = (H_A \otimes
H_B \otimes \C^2)^3, \textup{ and}
\end{eqnarray*}
\begin{eqnarray*}
\Dirac = \begin{bmatrix} \Dirac_{1} & 0 & 0 \\[1ex] 
0 & 0 &   \Dirac_2 - i \Dirac_{3}\\[1ex]  
0  & \Dirac_2  +
i \Dirac_{3} &0 \end{bmatrix}.
\end{eqnarray*}
Then $(\mathcal{E}, H, \Dirac)$, represented
via $\Pi$, defines a spectral triple on $E$. 
Moreover, the spectral dimension of this spectral triple is computed by
the identity
\begin{eqnarray*}
s_0(\mathcal{E}, H, \Dirac) = s_0(\Alg,H_A,\Dirac_A) + s_0(\Balg,H_B,\Dirac_B).
\end{eqnarray*}
\end{customthm}

\vspace{5mm}

The Dirac operator $\Dirac$ of the spectral triple defines a Lipschitz seminorm $L_{\Dirac}$ which in turn defines in  good cases a metric on the state space of the C$^*$-algebra. We address the question of whether our spectral triples satisfy Rieffel's metric condition, which is a necessary and sufficient condition for the metric on the state space to induce the weak-$*$-topology (cf. Prop.\ref{metric condition} and the definition thereafter). In this case the spectral triple together with the Lipschitz seminorm is called a spectral metric space. We show that under our natural assumptions this is always the case.

\begin{customthm}{\ref{ConstExt2}}
Let $A$ and $B$ be unital C$^*$-algebras and suppose $E$ arises as the short exact sequence
(\ref{extension}). Suppose further that there exists spectral triples $(\Alg,H_A,\Dirac_A)$ on $A$ and
$(\Balg,H_B,\Dirac_B)$ on $B$, represented via $\pi_A$ and $\pi_B$ respectively, and an orthogonal projection $P \in B(H_A)$ such that 
$(\Alg, H_A, \Dirac_A, P)$ is of Toeplitz type and $P$-injective. If the spectral triples $(\Alg,H_A,\Dirac_A)$ and
$(\Balg,H_B,\Dirac_B)$ satisfy Rieffel's metric condition then so does the spectral triple $(\mathcal{E}, H, \Dirac)$ so that $(\mathcal{E}, L_{\Dirac})$ is a spectral metric space.
\end{customthm}

We go on to show that there are numerous examples of C$^*$-algebra extensions which can be given the structure of
a spectral metric space. Our main focus is the single-parameter noncommutative (quantum) spheres $C(S_q^n)$ for $n
\geq 2$, which can be iteratively defined as C$^*$-algebra extensions of smaller noncommutative spheres. We shall
specifically study the cases $n = 2$ (the equatorial Podle{\'s} spheres) and $n = 3$ (the quantum $\textup{SU}_q(2)$ group) and merely comment on how these two examples can be used to study their higher dimensional counterparts.

The noncommutative spheres have garnered a lot of attention in the literature as examples of noncommutative
manifolds and many spectral triples have been suggested (e.g. \cite{CPP1}, \cite{CPP3},  \cite{DDLW}), though most of these from a very different perspective to ours, for example by looking at the representation theory of the ordinary $\textup{SU}(2)$ group and focusing on those triples which behave equivariantly with respect to the group co-action. 
%We do not see any kind of connection between these examples and ours at present.

We remark that Chakraborty and Pal also considered the question of finding Lip-metrics starting from given ones on the ideal and quotient for
extensions of the same type as ours \cite{CP1}. However, our goal is to construct spectral triples, rather than compact quantum metric spaces. Our spectral triples give rise to Lip-metrics with properties similar to theirs but we have existence results for Dirac operators and our constructions seem to be fairly different.

\vspace{3mm}

\textbf{Acknowledgement:} We are very much indebt to the anonymous referees and would like to thank them for many very helpful comments which improved the paper considerably.

\section{A review of spectral triples and quantum metric spaces}\label{review}

\subsection{Spectral Triples.}\label{spectriples}

We begin with a short exposition of spectral triples. 
For more information and context, we recommend the articles \cite{Ren}, \cite{Var} and \cite{LRV} which provide an excellent exposition of the theory and motivation behind spectral triples. We remind the reader that $C^*$-algebras and Hilbert spaces are assumed to be separable throughout this article.

\begin{definition} \label{deftriple} Let $A$ be a separable C$^*$-algebra. A spectral triple $(\Alg,\Hil,\Dirac)$
on $A$ is given by a faithful $^*$-representation $\pi: A
\mapsto B(\Hil)$ on a Hilbert space $\Hil$, a dense $^*$-subalgebra $\Alg \subseteq A$ and a linear densely defined
unbounded self-adjoint operator $\Dirac$ on
$\Hil$ such that
\begin{enumerate}
\item $\pi(\Alg) \textup{dom} (\Dirac) \subseteq \textup{dom} (\Dirac)$ and $[\Dirac, \pi(a)]: \textup{dom}(\Dirac)
    \to \Hil$ extends to a bounded operator for each $a \in \Alg$ and
\item $\pi(a)(I + \Dirac^2)^{-1}$ is a compact operator for each $a \in A$.
\end{enumerate}
\end{definition}
Unlike the majority of definitions provided in the literature, we do not make the assumption that the C$^*$-algebra $A$ is unital, or indeed that the representation over $\Hil$ is nondegenerate. Using faithfulness on the other hand we may identify $A$ with $\pi(A)$ and therefore omit the representation from notation, in particular writing $a\xi$ for $\pi(a)\xi$.

For a spectral triple on $A$, it is sometimes convenient to study the \textit{maximal Lipschitz subalgebra},
$C^1(A)$ which comprises those elements $a \in A$ such that $\pi(a) (\textup{dom}(\Dirac)) = a( \textup{dom}(\Dirac)) \subseteq \textup{dom}(\Dirac)$, the operator $[\Dirac, \pi(a)]: \textup{dom} (\Dirac) \to \Hil$ is closable and $\delta_{\Dirac}(a) := \textup{cl}[\Dirac, \pi(a)]$ is a bounded operator in $B(\Hil)$. It is an analogue of the algebra of Lipschitz continuous functions on a Riemannian spin$^C$ manifold. 
It is not immediately obvious, as one might think, that $C^1(A)$ is a $^*$-algebra. This  follows
from the fact that if any two elements $a$ and
$b$ leave the domain of $\Dirac$ invariant then so does $ab$ and the operator $[\Dirac,ab] = [\Dirac,a]b +
a[\Dirac,b]$, defined on $\textup{dom}(\Dirac)$, extends
to a bounded operator in $B(\Hil)$. It is less clear that  $C^1(A)$ is closed under involution. We are grateful to  Christensen for pointing out the following way to show this. In \cite{Chr}, he shows that the above
condition can be replaced by requiring the sesquilinear form $S([\Dirac,a])$, defined on $\textup{dom}(\Dirac) \times \textup{dom}(\Dirac)$ as
\begin{eqnarray*} S([\Dirac,a])(\xi,\eta) :=
\ip{a\xi}{\Dirac \eta} - \ip{a\Dirac\xi}{\eta}, \;\xi,\eta \in \textup{dom}(\Dirac),
\end{eqnarray*}
to be defined and bounded. The equality $S([\Dirac,a^*])(\xi,\eta) = -S([\Dirac,a](\eta,\xi))^*$ ensures that
$C^1(A)$ is closed under involution. It is well known that $C^1(A)$ becomes an operator algebra when equipped with
the norm $\|a\|_1 := \|\pi(a)\| + \|[\Dirac,\pi(a)]\|$ and viewed as a concrete subalgebra of the bounded
operators on the first Sobolev space, $\Hil_1 := \textup{dom}(\Dirac)$, of $\Hil$, the latter  being a Hilbert space with
respect to the inner product $\ip{\eta_1}{\eta_2}_1 := \ip{\eta_1}{\eta_2} + \ip{\Dirac \eta_1}{\Dirac \eta_2}$.
Depending on the context, it can be useful to think of $C^1(A)$ as either a dense $^*$-subalgebra of $A$ or as a
Banach algebra in its own right.

Recall that a spectral triple $(\Alg, \Hil, \Dirac)$ on a \textit{unital C$^*$-algebra} is called
\textit{$p$-summable}, (sometimes $(p,\infty)$-summability), where $p \in (0,\infty)$, if 
$(I + \Dirac^2)^{-p/2} \in B(\Hil)$ lies in the Dixmier class which is strictly larger than the trace class. 
\begin{definition}
The \textit{spectral dimension} of $(\Alg,\Hil,\Dirac)$, defined on a unital C$^*$-algebra $A$, is
given by
\begin{eqnarray*}
s_0(\Alg,\Hil,\Dirac) = \inf\{ p \in (0,\infty) : \; (\Alg,\Hil,\Dirac) \textup{ is } p-\textup{summable} \} 
\end{eqnarray*}
It can be shown that 
\begin{eqnarray*}
s_0(\Alg,\Hil,\Dirac) =\inf\{ p \in (0,\infty): \;\textup{Tr}(I + \Dirac^2)^{-p/2} < \infty\}.
\end{eqnarray*}
Here $\textup{Tr}$ is the usual unbounded trace on $B(H)$. 
\end{definition}
We will often write $s_0(\Dirac)$ instead of $s_0(\Alg,\Hil,\Dirac)$ and employ this notation also for the summability of an unbounded essentially self-adjoint operator. 

 Summability can also be defined for spectral triples of non-unital C$^*$-algebras, as advocated by Rennie in \cite{Ren1}.

Because of the relationship between spectral triples and Fredholm modules in K-homology, spectral triples are
often distinguished into \textit{odd} and
\textit{even} varieties:
\begin{definition} A spectral triple on $A$ is called \textit{graded} or \textit{even} if there exists an operator $\gamma \in B(\Hil)$ such that
$\gamma^2 = \textup{id}$, $\gamma \pi(a) = \pi(a) \gamma$ for each $a \in A$ and $\gamma \Dirac = -
\Dirac \gamma$. Otherwise it will be called
\textit{ungraded} or odd.
\end{definition}
Stated in a different way, an even spectral triple is one which can be formally represented via a direct sum representation over an orthogonal direct sum of Hilbert spaces $\Hil = \Hil^+ \oplus \Hil^-$ over which $\pi$ and
$\Dirac$ decompose as
\begin{eqnarray*}
\pi = \begin{bmatrix} \pi_0^+ & 0 \\ 0 & \pi_0^- \end{bmatrix},  
\;\;\;\;\;
\Dirac = 
\begin{bmatrix} 0 & \Dirac^- \\ 
\Dirac^+ & 0 \end{bmatrix},
\;\;\;\;\; 
\gamma = 
\begin{bmatrix} 1 & 0 \\ 
0 & -1 \end{bmatrix}.
\end{eqnarray*}

\subsection{Compact quantum metric spaces.}
One of the most interesting aspects of spectral triples in differential geometry is the possibility to recover the metric information of the manifold from the spectral triple.  Connes \cite{Con2} generalises this observation by showing that a spectral triple $(\Alg, \Hil,\Dirac)$ on a
C$^*$-algebra $A$ defines an \textit{extended metric} (i.e. allowing the metric to take the value $\infty$) 
$d_C: S(A) \times S(A) \to [0,\infty]$ on
the state space $S(A)$ of $A$, by the formula
\begin{eqnarray*}
d_C(\omega_1, \omega_2) := \sup\{|\omega_1(a) - \omega_2(a)|: \;a = a^* \in \Alg,
\;\|[\Dirac,\pi(a)]\| \leq 1\}
\end{eqnarray*}
Connes' metric $d_C$ in general depends on the algebra $\Alg$, so it is often better to write $d_{\Alg}$ to
stress
this dependence. The motivating example is prescribed by the Dirac triple on a connected spin$^c$ manifold
$\mathcal{M}$ defined on the dense subalgebra of
``C$^{\infty}$-functions" for which $\|[\Dirac, f]\| = \|df\|$. The restriction of Connes' metric to the point
evaluation measures $d_C(p_x, p_y)$ then coincides
with the geodesic metric $d_{\gamma}(x,y)$ along $\mathcal{M}$.

In \cite{Ri1}, \cite{Ri2} and \cite{Ri3}, Rieffel considered the more general setting of Lipschitz seminorms, which can be viewed as a generalisation of
metric spaces, or Lipschitz functions, to order-unit
spaces. The theory is based on the observation of Kantorovich and Rubinstein, who demonstrated that a metric on a
compact topological space can be extended
naturally to the set of probability measures on that space. Recent work by Latr{\'e}moli{\`e}re in \cite{Lat1} and \cite{Lat2} has extended much
of this work to the setting of non-unital
C$^*$-algebras.

In the context of this paper, a \textit{Lipschitz seminorm} on a separable C$^*$-algebra $A$ is a seminorm $L:
\Alg \to \R^+$ defined on a dense subalgebra
$\Alg$ which is closed under involution with the property $L(a^*) = L(a)$ for each $a \in \Alg$ and also $L(1) =
0$ whenever $\Alg$ is unital. We say that a Lipschitz seminorm $L$ is
\textit{nondegenerate} if the set $\{a \in \Alg: \;L(a) = 0\}$ is trivial or contains only multiples of the
identity when $A$ is unital. As pointed out in \cite{HSWZ}, nondegeneracy of $L$ is independent of the choice of the dense subalgebra $\Alg$, but it should be stressed that many of the properties of the Lipschitz seminorm  will depend on $\Alg$.

A Lipschitz seminorm on $A$ determines an extended metric $d_{\Alg,L}$ on $S(A)$ (occasionally written
$d_{\Alg}$, or $d_L$) in a way which provides a noncommutative analogue of the Monge-Kantorovich distance when $A$
is commutative. The metric is given by
\begin{eqnarray}
\label{NoncKan} d_{\Alg,L}(\omega_1,\omega_2) := \sup \{|\omega_1(a) - \omega_2(a)|: \;a \in \Alg, \;L(a) \leq
1\}.
\end{eqnarray}

Conversely a metric $d$ on $S(A)$ defines a nondegenerate seminorm $L_d$ on $A$ via
\begin{eqnarray}
\label{Recov} L_d(a) := \sup \bigg{ \{ } \frac{|\mu(a) -
\nu(a)|}{d(\mu,\nu)}: \; \mu,\nu \in S(A),\;\mu \neq \nu \bigg{ \} }.
\end{eqnarray}
If $L$ is a Lipschitz seminorm on $A$, so is $L_{d_{\Alg,L}}$. When the
Lipschitz seminorm $L$ is \textit{lower semicontinuous}, so that the set $\{a \in \Alg: \; L(a) \leq r\}$ is
closed in $A$ for any and hence all $r > 0$, then $L =
L_{d_{\Alg,L}}$. We shall further call $L$ closed if $L$ is lower semicontinuous and $\Alg =
\textup{dom}(L_{d_{\Alg,L}})$. Hence, starting from a lower
semicontinuous seminorm, the above procedure can be used to extend $L$ to a closed seminorm. All these
observations are well known in the case when $\Alg$ is
unital and the procedure of replacing $\Alg$ with its unitisation $\overline{\Alg} = \Alg \oplus \C I$ and
introducing the new seminorm $\overline{L}(a,\lambda) :=
L(a)$ can easily be used to generalise these results to the non-unital case.

\begin{proposition} (\cite{RV})
Let $(\Alg, \Hil, \Dirac)$ be a spectral triple over a C$^*$-algebra $A$ with faithful representation
$\pi: A \to B(\Hil)$ such that $[\Dirac,\pi(a)] = 0 \iff a \in \C I_A$. Then $L_{\Dirac}(a) := \|[\Dirac, \pi(a)]\|$, defines a
lower semicontinuous Lipschitz seminorm on $\Alg$,
which is closed if and only if $\Alg = C^1(A)$. If the representation $\pi$ is nondegenerate and the spectral
triple comes with a cyclic vector $\xi$ for $(A,\pi)$
such that $\ker{\Dirac} = \C \xi$ then Connes' extended metric on $S(A)$ is a metric.
\end{proposition}

Rieffel addresses the question of whether a metric induced by a Lipschitz seminorm on a unital separable C$^*$-algebra (or order-unit space) has finite diameter and whether it induces the weak-$*$-topology of $S(A)$ which is a compact metrisable Hausdorff space in this situation. 
To state his result we introduce some notation: for a given Lipschitz seminorm $L$ on $\Alg$, define 
$$
B_L(\Alg) = \{ a \in \Alg : L(a) \leq 1\}, \;\;\;\;\;
\widetilde{B}_L(\Alg) := \{\tilde{a} \in \Alg / \C I : L(a) \leq 1\}$$ 
(note that $L$ passes to the quotient $\Alg / \C I $) and
$$B_{1,L}(\Alg) := \{ a \in B_L(\Alg) : \|a\| \leq 1\}=B_L(\Alg) \cap \overline{\textup{B}}_A,$$
where $\overline{\textup{B}}_A$ is the closed unit ball in $A$.

\begin{proposition}{(\cite{Ri1} 1.8 and 1.9, \cite{Lat1})}\label{metric condition} Let $A$ be a unital C$^*$-algebra, equipped with a nondegenerate Lipschitz seminorm $L$ on a     dense $*$-subalgebra $\Alg$ of $A$. Then:
\begin{enumerate}
\item Equation (\ref{NoncKan}) determines a metric $d_{L,\Alg}$ of finite diameter if and only if $\widetilde{B}_L(\Alg)
    \subseteq A / \C I$ is
    norm-bounded, and further diam$(\widetilde{B}_L(\Alg), \| \cdot \|_{A / \C I}) \leq r$ if and only if
    diam$(S(A),d_L) \leq 2r$, for each $r > 0$.
\item $d_{\Alg,L}$ metrises the weak-$*$-topology of $S(A)$ if and only if the following conditions are satisfied.
\begin{enumerate}
\item $d_{\Alg,L}$ has finite diameter.
\item  $B_{1,L}(\Alg)  \subseteq  A$ is totally bounded in norm.
\end{enumerate}
\end{enumerate}
\end{proposition}

We will refer to the conditions 2.(a) and 2.(b) in Prop.\ref{metric condition} as \textit{Rieffel's metric conditions} or just \textit{metric conditions}.
The situation when $A$ is non-unital is rather more complicated, but Latr{\'e}moli{\`e}re (\cite{Lat1})  shows that, provided one works only with Lipschitz seminorms which give a metric on
$S(A)$ with finite diameter, things are not too complicated. For this case he provides conditions similar to the ones in the preceeding Proposition \ref{metric condition}
which characterise those seminorms that induce the weak-$*$-topology on $S(A)$.

\begin{definition}\label{quantum metric space}
Let $A$ be a separable C$^*$-algebra equipped with a Lipschitz seminorm $L$ on a suitable
dense subalgebra $\Alg$ with the property that $d_{\Alg,L}$ determines a metric of finite diameter inducing the weak-$*$-topology of $S(A)$. Then the pair $(\Alg, L)$  is called a
\textit{quantum metric space} (or  \textit{compact quantum metric space} when $\Alg$ is unital).
\end{definition}
       
Thus $(\Alg, L)$ with $\Alg$ unital will be compact quantum metric space if and only if $(\Alg, d_{\Alg,L})$ satisfies Rieffel's metric conditions.

The final definition is motivated by a similar definition in \cite{BMR} which we will follow in this paper.

\begin{definition}\label{spectral metric space}
Let $(\Alg, H, \Dirac)$ be a spectral triple with corresponding Lipschitz seminorm $L_{\Dirac}$. If $(\Alg, L_{\Dirac})$ is a \textit{quantum metric space}, then $(\Alg, H, \Dirac)$ (or $(\Alg, L_{\Dirac})$) is called a \textit{spectral metric space}. 
\end{definition}

\section{Extensions and Kasparov's KK-Theory.}\label{ext}

In this section we recall and develop some background in KK-theory related to extensions. Further information can be found in Kasparov's seminal paper \cite{Kas3} and in  \cite{Bla}, \cite{HR}.

For a separable C$^*$-algebra $A$, we write $\ell_2(A)$ to mean the Hilbert bimodule $A$ consisting of sequences of the form $(a_n)_{n \in \N}$ such that $\sum_n a_n^* a_n$ converges in norm, equipped with the inner product $\ip{(a_n)}{(b_n)} := \sum_n a_n^* b_n$. We write
$\mathcal{L}_A$ to mean the set of adjointable right $A$-linear operators on $\ell_2(A)$, which becomes a C$^*$-algebra when equipped with the operator norm. We denote by $\mathcal{K}_A$ the C$^*$-subalgebra of $\mathcal{L}_A$ consisting of the closed linear span of operators of the form
$\theta_{x,y}(z) = x\ip{y}{z}, \;x,y,z \in
\ell_2(A)$. Then $\mathcal{K}_A$ is an ideal in $\mathcal{L}_A$ and is isomorphic to the spatial tensor product, $\mathcal{K} \otimes A$, of $A$ by the compact operators on a separable infinite dimensional  Hilbert space. The algebra $\mathcal{L}_A$ is isomorphic to $\mathcal{M}(\mathcal{K} \otimes A)$.
We denote the quotient $\mathcal{L}_A
/ \mathcal{K}_A$ by  $ \mathcal{Q}_A$ and will also use the notation  $\mathcal{L}$ for $B(\ell^2)$ and $\mathcal{Q}$ for the quotient $B(\ell^2)/\mathcal{K}$. 

\subsection{Background and set-up}
The extensions in this article are unital short exact sequences of separable C$^*$-algebras of the form,
\begin{eqnarray}
\label{ExtStableIdeal} \xymatrix{ 0 \ar[r] & \mathcal{K} \otimes B \ar[r]^{\iota} & E \ar[r]^{\sigma} & A \ar[r] &
0}. \label{TopEx}
\end{eqnarray}
Recall that this means $\iota$ is an injective $^*$-homomorphism and regarded as an inclusion map, $\sigma$ is a
surjective $^*$-homomorphism and $\im(\iota) = \ker (\sigma)$. We will always assume that the C$^*$-algebras $A$ and $B$
are unital and that $\mathcal{K} \otimes B$ is the stabilisation of $B$ by compact operators on a
separable, infinite dimensional Hilbert space. Additionally we will always require the ideal $\mathcal{K} \otimes
B$ to be \textit{essential}, i.e. it has non-zero intersection with any other ideal $I \subseteq E$.

The \textit{Busby invariant} of (\ref{ExtStableIdeal}) is a $^*$-homomorphism $\psi: A \to \mathcal{L}_B
/ \mathcal{K}_B =: \mathcal{Q}_B$. The $^*$-homomorphism  $\psi$ can be regarded as  a characteristic of the extension itself, since the
extension can be recovered from $\psi$, up to isomorphism, as the pullback C$^*$-algebra,
\begin{eqnarray}
E \cong \mathcal{L}_B \oplus_{(q_B, \psi)} A := \{(x,a) \in \mathcal{L}_B \oplus A: \;q_B(x) = \psi(a)\} \label{pull-back}
\end{eqnarray}
(here, $q_B: \mathcal{L}_B \to \mathcal{Q}_B$ is the quotient map). The assumptions above imply that $\psi$ and consequently the map $\pi: E \to \mathcal{L}_B, \;\pi(x,a) = x$ is injective. The maps fit together in the commuting diagram
\begin{eqnarray}\label{commuting diagram}
\xymatrix{ 0 \ar[r] & \mathcal{K}_B \ar@{^{(}->}[r]^{\iota} \ar[d]^{\pi|_{\overline{B}}} & E \ar[r]^{\sigma}
\ar[d]^{\pi} & A \ar[r] \ar[d]^{\psi}
& 0\\ 0 \ar[r] & \mathcal{K}_B \ar@{^{(}->}[r]^{\iota_B} & \mathcal{L}_B \ar[r]^{q_B} & \mathcal{Q}_B \ar[r] &
0.}
\end{eqnarray}

We do not consider all such extensions, but restrict our attention to the situation in which the Busby invariant $\psi$ admits a unital completely positive lift, i.e., there is a unital completely positive map $s: A \to \mathcal{L}_B$ such that $q_B \circ s = \psi$. This is equivalent to the existence of a ucp lift of $\sigma$. Such extensions are called \textit{semisplit}.
A well known application of the Kasparov-Stinespring Theorem shows that, in this setting, there is a faithful representation $\rho: A \to M_2(\mathcal{L}_B) \cong \mathcal{L}_B$ and an orthogonal projection $P \in M_2(\mathcal{L}_B) \cong \mathcal{L}_B$ such that $[P, \rho(a)] \in M_2(\mathcal{K}_B)$ and $\rho_{11}(a) = s(a)$
for each $a \in A$, where,
\begin{eqnarray*}
\rho = \begin{pmatrix} \rho_{11} & \rho_{12} \\ \rho_{21} & \rho_{22} \end{pmatrix}; \;\;P = \begin{pmatrix} 1 & 0
\\ 0
& 0 \end{pmatrix}.
\end{eqnarray*}
We call $(\rho, P)$ a Stinespring dilation of $s: A \to \mathcal{L}_B \cong \mathcal{M}(\mathcal{K} \otimes B)$.
The existence of such a map is not automatic, unless $A$ is a nuclear C$^*$-algebra in which case the existence of a completely positive lift follows from the Choi-Effros lifting theorem. 

The Stinespring dilation $(P, \rho)$ can be used to define a Kasparov cycle $\psi^*$ which is the element of $KK^1(A,B)$
represented by the triple $(\ell_2 (B) \oplus \ell_2(B), \rho, 2P-1)=(\ell_2 (B) , \rho, 2P-1) $. A well known result of Kasparov (\cite{Kas3}) says that there is a six-term
exact sequence in both K-theory and K-homology. In the  case of K-homology the sequence has the  form
\begin{eqnarray}
\label{SixTerm} \xymatrix{K^0(A) \ar[r]^{\sigma^*} & K^0(E)  \ar[r]^{\iota^*} & K^0(B) \ar[d]^{\delta_0^*} \\
K^1(B)
\ar[u]^{\delta_1^*} & \ar[l]^{\iota^*} K^1(E) & \ar[l]^{\sigma^*} K^1(A),  }
\end{eqnarray}
 where the boundary maps are defined by taking the internal Kasparov product with $\psi^*$.

\subsection{Toeplitz type extensions and KK-theory}

In this section we discuss a characterisation of Toeplitz type extensions showing that they form a large class. Moreover, we will introduce the representations of the extension algebra which are relevant in order to define our spectral triple on the extension algebra.

In what follows, we will assume $B$ is unital and we shall let $j: \mathcal{L} \to \mathcal{L}_B$,
$\bar{\jmath}: \mathcal{Q} \to \mathcal{Q}_B$, $q: \mathcal{L} \to \mathcal{Q}$ and $q_B: \mathcal{L}_B
\mapsto \mathcal{Q}_B$ be the natural maps, so that $q_B \circ j = \bar{\jmath} \circ q$. The main result of this section is contained in the following:

\begin{proposition}\label{Toeplitz-characterisation}
Given an extension (\ref{ExtStableIdeal}), where $A$ and $B$ are separable C$^*$-algebras and $\psi: A \to \mathcal{Q}_B$ is the Busby invariant of this extension, the following are equivalent.
\begin{enumerate}
\item There is a $^*$-homomorphism $\psi_0: A \to \mathcal{Q}$ such that $\bar{\jmath} \circ \psi_0 = \psi$.
\item There is a C$^*$-algebra $E_0$, an injective $^*$-homomorphism $r: E_0 \to E$, an injective $^*$-homomorphism $\pi_0: E_0 \to \mathcal{L}$ and a surjective $^*$-homomorphism $\sigma_0: E_0 \to A$ such that the following diagrams commute:
\begin{eqnarray}
\xymatrix{ 0 \ar[r] & \mathcal{K} \ar@{^{(}->}[r]^{\iota_0} \ar[d]^{r|_{\mathcal{K}}} & E_0 \ar[r]^{\sigma_0}
\ar[d]^{r} & A \ar[r] \ar@{=}[d] &
0\\ 0 \ar[r] & \mathcal{K}_B \ar@{^{(}->}[r]^{\iota} & E \ar[r]^{\sigma} & A \ar[r] & 0}
\label{character1}
\end{eqnarray}
\begin{eqnarray}
\xymatrix{ E_0 \ar[r]^{\pi_0} \ar[d]^{r} &
\mathcal{L} \ar[d]^{j} \\ E \ar[r]^{\pi} & \mathcal{L}_B.}
\label{ch2}
\end{eqnarray}
\end{enumerate}
\end{proposition}
\begin{proof}
$(1) \implies (2)$: Starting from a
homomorphism $\psi_0: A \to \mathcal{Q}$ as above, we can define $E_0$ as the pullback C$^*$-algebra
\begin{eqnarray*}
E_0 := B(\ell^2) \oplus_{(q, \psi_0)} A :=
\{(x,a) \in B(\ell^2) \oplus A : q(x) = \psi_0(a)\}.
\end{eqnarray*}
There are  natural maps $\pi_0: E_0 \to B(\ell_2)$ and $\sigma_0: E_0 \to A$. Similarly, as pointed out in (\ref{pull-back}) $E$ is given as a pullback $E= \mathcal{L}_B \oplus_{(q_B,\psi)} A$, so that we also have natural maps $\pi : E \to \mathcal{L}_B$ and $\sigma : E \to A$. Notice that $\ker(\sigma_0) = \ker (q \circ \pi_0)$ is isomorphic to the algebra of compact operators. The map $r: E_0 \to E$ can be defined by $r((x,a)) := (j \circ \pi_0(x), a)$, which is injective and thus establishes the first part of the proof.

\medskip

\noindent
$(2) \implies (1)$: Let $\psi_0$ be the map which is given by
\begin{eqnarray*}
\psi_0(\sigma_0(e)) = q(\pi_0(e))).
\end{eqnarray*}
This map is well defined: if $e_1, e_2 \in E_0$ are such that $\sigma_0(e_1) = \sigma_0(e_2)$ then $e_1 - e_2$ is a compact operator, so that $q(\pi_0(e_1 - e_2)))$ vanishes. For any $e \in
E_0$, we have
\begin{eqnarray*}
(\bar{\jmath} \circ \psi_0 \circ \sigma_0)(e) = (\bar{\jmath} \circ q \circ \pi_0)(e) && \textup{(by definition of $\psi_0$)} \\
 = (q_B \circ j \circ \pi_0)(e) && \textup{(since $q_B \circ j = \bar{\jmath} \circ q$)}  \\
 = (q_B \circ \pi \circ r)(e) && \textup{(from diagram (\ref{ch2}))} \\
 = (\psi \circ \sigma \circ r)(e) && \textup{(from diagram (\ref{commuting diagram}))} \\
 = (\psi \circ \sigma_0)(e)&& \textup{(from diagram (\ref{character1})}.
\end{eqnarray*}
Since $\sigma_0$ is surjective, $\bar{\jmath} \circ \psi_0 = \psi$, completing the proof.
\end{proof}

It is clear that in the setting of the last Proposition \ref{Toeplitz-characterisation}, the map $\psi_0$ is injective if and only if $\psi$ is injective, so that we can assume the extension corresponding to the top row of (\ref{commuting diagram}) is essential. If $s: A \to
\mathcal{L}$ is a completely positive lift for $\psi_0$ then $j \circ s$ is a completely positive lift for $\psi$.

To apply this to  our Toeplitz type extensions recall (Def.\ref{Toeplitz type}) that an extension (\ref{extension}) is of Toeplitz type if  $\pi_A: A \to B(H_A)$  and $\pi_B: B \to B(H_B)$ are faithful representations with $[P, \pi_A(a)] \in \mathcal{K}(H_A)$ for each $a \in A$, $P \pi_A(a)P  \; \cap \;\mathcal{K}(P H_A) =    \{0\}$,
and $E$ is isomorphic to the subalgebra of $B(H_A \otimes H_B)$ generated by $\mathcal{K}(PH_A) \otimes    \pi_B(B)$ and $P\pi_A(A)P \otimes \C I_B$.  Thus, omitting the representations, the extension is of the form
\begin{eqnarray}
\xymatrix{ 0 \ar[r] & \mathcal{K} \otimes B \ar[r]^{} &  \mathcal{K}(PH_A) \otimes B + PAP \otimes \C I_B \ar[r]^{} & A \ar[r] & 0.}
\label{E}
\end{eqnarray} 
$(\pi_A, \pi_B, P)$ is called the corresponding Toeplitz triple. From this extension we obtain the extension
\begin{eqnarray}
\xymatrix{ 0 \ar[r] & \mathcal{K}(PH_A) \ar[r]^{} & E_0 \ar[r] & A \ar[r] & 0}.
\label{E_0}
\end{eqnarray}
of $A$ by $\mathcal{K}$,
where $E_0= P\pi_A(A) P + \mathcal{K} (PH_A)$. There is a natural inclusion map $r: E_0 \hookrightarrow E$. Moreover, $E_0$ embedds naturally into $\mathcal{L}\cong B(H_A)$ defining a degenerate (i.e. non-unital) but faithful representation $\pi_0 : E_0 \to \mathcal{L}$.

Similarly, there is a degenerate but faithful representation $\pi : E \to B(H_A \otimes H_B)$ given by its very definition as a subalgebra. Since $H_A$ is separable and infinite dimensional we have 
$\mathcal{L}_B \cong \mathcal{M}(\mathcal{K} \otimes B) \cong \mathcal{M}(\mathcal{K}(H_A) \otimes \pi_B (B)) \subseteq B(H_A \otimes H_B)$. We can therefore think of $\pi$ as a representation
$\pi : E \to \mathcal{L}_B$. There is a natural inclusion  $r:E_0  \to E$ (using that $B$ is unital) such that the diagrams in 
Prop.\ref{Toeplitz-characterisation}.(2) commute. 

Note that the Busby invariants for the extension $E$ and $E_0$ are given by 
$$
\psi(a)= q_B((P\otimes I ) (\pi_A(a) \otimes I) (P \otimes I))
$$ and 
$$
\psi_0 (a) = q(P \pi_A(a) P), 
$$
which implies that  $\bar{\jmath} \circ \psi_0 = \psi$. Thus starting from a Toeplitz type extension $E$ of $A$ by $\mathcal{K} \otimes B$ we found  an extension $E_0$ of $A$ by $\mathcal{K}$ satisfying the conditions (1) and, hence, (2) of Prop.\ref{Toeplitz-characterisation}.

We mention that with this interpretation  $s(a):= P \pi_A(a)P \otimes 1$ can be regarded as a completely positive map $s : A \to \mathcal{M}(\mathcal{K}(PH_A) \otimes B)$ such that $q \circ s : A \to \mathcal{Q} (\mathcal{K} (P H_A) \otimes B)$ is the Busby invariant of the extension. Hence $s$ is a cp-lift of the extension and  $(P \otimes 1, \pi_A \otimes 1)$ can be regarded as  Stinespring dilation of the semisplit extension (\ref{E}).

\begin{corollary} \label{tensorial}
Let $E$ be a Toeplitz type extension (\ref{E}) with corresponding extension (\ref{E_0}). Then, with $\pi, \pi_0, r, \psi, \psi_0$ as defined above,  condition (1) and (2) of Prop.\ref{Toeplitz-characterisation} are satisfied.

Conversely, given an essential semisplit extension  
\begin{eqnarray*}
\xymatrix{ 0 \ar[r] & \mathcal{K}  \ar[r]^{\iota} & E_0 \ar[r] & A \ar[r] & 0}.
\end{eqnarray*}
then
there exists faithful representations $\pi_A : A \to B(H_A)$ and $\pi_B : B \to B(H_B)$ and $P \in B(H_A)$ an infinite dimensional projection such that $E_0 \cong P\pi_A(A) P + \mathcal{K} (PH_A)$ and if we define $E$ by $E= \mathcal{K}(PH_A) \otimes B + PAP \otimes \C I_B$ and $\pi, \pi_0, r, \psi, \psi_0$ as before then (1) and (2) of Prop.\ref{Toeplitz-characterisation} are satisfied.
\end{corollary}

\begin{proof}
The first part follows from the discussion preceeding the Corollary. For the second part it is known (\cite{HR}, 2.7.10) that the required representation $\pi_A $ and projections can be found for every semisplit extension of the form $0 \to \mathcal{K} \to E_0 \to A \to 0$ (it is given by the Stinespring dilation we decribed.) Once we have that we can use any faithful representation $\pi_B : B \to B(H_B)$ and  define $E$ and $\pi, \pi_0, r, \psi, \psi_0$ as before satisfying the required identities.
\end{proof}

When a Toeplitz triple exists, we have our  $^*$-homomorphism $\pi : E \to B(H_A \otimes H_B)$ given on the generators via
\begin{eqnarray*}
\pi(k \otimes b)(\eta \otimes \nu) &:= &k\eta \otimes \pi_B(b)\nu, \\
\pi(PaP \otimes I_B)(\eta \otimes \nu) &:=&
P\pi_A(a)P \eta \otimes \nu, 
\end{eqnarray*}
where $a \in A$ and $k \in  \mathcal{K}(PH_A)$ which is faithful but degenerate (i.e. not unital). We also have another representation $\pi_{\sigma}: E \to B(H_A \otimes H_B)$ given by 
$$
\pi_{\sigma} = \pi_A \circ \sigma \otimes 1,
$$
where $\sigma : E \to A$ is the quotient map in the extension, given by 
$$
\sigma(k \otimes b) := 0, \textup{ and } \sigma(PaP \otimes I_B) := a,
$$
where $a \in A$, $b \in B$ and $k \in \mathcal{K}(PH_A)$. 
The representation $\pi_{\sigma}$ is non-degenerate (unital) but not faithful.

%For later use we define the representations $\Pi_1, \Pi_2: E \to B(H)$  by
%\begin{eqnarray} \label{DirectSumRep}
%\Pi_1 := \pi_{\sigma} \oplus \pi_{\sigma}, \;\;\Pi_2 := \pi \oplus \pi_{\sigma}, \;\;H := H_A %\otimes H_B \otimes
%\C^2.
%\end{eqnarray}

\section{Construction of the spectral triple.}\label{const}

We will now begin to describe the steps needed to construct a spectral triple on an extension
(\ref{ExtStableIdeal}) of Toeplitz type with Toeplitz triple $(\pi_A,\pi_B,P)$ introduced earlier in Def.\ref{Toeplitz type}.

\subsection{Smoothness criteria.}
In this section we want to discuss the smoothness conditions in Def.\ref{Toeplitz type} which we need for our main result. As stated in Def.\ref{Toeplitz type} given a Toeplitz type extension 
\begin{eqnarray*}
\xymatrix{ 0 \ar[r] & \mathcal{K} \otimes B \ar[r]^{} &  \mathcal{K}(PH_A) \otimes B + PAP \otimes \C I_B \ar[r]^{} & A \ar[r] & 0}
\end{eqnarray*}
and a spectral triple $(\Alg, H_A, \Dirac_A)$ on $A$ we say that the quadruple $(\Alg, H_A, \Dirac_A, P)$  is of Toeplitz type if  
\begin{enumerate}
\item
$P$ and $\Dirac_A$ commute,
\item
$[P, \pi_A (\Alg)] \subseteq \mathcal{C}(H_A) $,
\end{enumerate}
where $\mathcal{C}(H_A)$ was discussed in the introduction and is formally defined below.

Condition (1) means that $P$ should leave the domain of $\Dirac_A$ invariant and
commute with each of the spectral projections of
$\Dirac_A$, so that the operator $[\Dirac_A, P]: \textup{dom}(\Dirac_A) \to H_A$ vanishes. Thus we can decompose $\textup{dom}(\Dirac_A)$ as an orthogonal direct sum of vector spaces,
\begin{eqnarray*}
(\textup{dom}(\Dirac_A) \cap PH_A ) \;\;\oplus \;\; (\textup{dom}(\Dirac_A) \cap (1-P)H_A),
\end{eqnarray*}
which are dense in $PH_A$ and $(1-P)H_A$ respectively. $\Dirac_A$ decomposes as a diagonal block matrix $\Dirac_A^p|_{P H_A} \oplus
\Dirac_A^q|_{(1-P) H_A}$ with respect to this decomposition, where $\Dirac_A^p := P \Dirac_A$ and $\Dirac_A^q := (1-P) \Dirac_A$.

To discuss the second condition we formally introduce the notion of differentiability for compact operators in the next definition.
\begin{definition} \label{D-diff}
We define the dense subalgebra of \textit{$\Dirac_A$-differentiable compacts}, $\mathcal{C}(H_A) \subseteq \mathcal{K}(H_A)$, to be the algebra of all compact operators $y \in \mathcal{K}(H_A)$ such that,
\begin{enumerate}
\item $y (\textup{dom}(\Dirac_A)) \subseteq \textup{dom}(\Dirac_A)$,
\item the operators $y\Dirac_A: \textup{dom}(\Dirac_A) \to H_A$ and $\Dirac_A y: \textup{dom}(\Dirac_A) \to H_A$ are closable,
\item the closures, $\textup{cl}(y \Dirac_A)$, $\textup{cl}(\Dirac_A y)$ respectively, are bounded operators.
\end{enumerate}
\end{definition}
\begin{remark}\label{motivation}
Our motivation for choosing the term $\Dirac_A$-differentiable compacts is based on the following observation: the same information as was given above can be used to write down an even spectral triple on the algebra of compact operators. It is given by the triple
\begin{eqnarray*}
\begin{pmatrix} \mathcal{C}(H_A ), & \textup{id} \oplus 0, & \Dirac := \begin{bmatrix} 0 & \Dirac_A
\\[1ex] \Dirac_A & 0 \end{bmatrix} \end{pmatrix}.
\end{eqnarray*}
We note that $\mathcal{C}(H_A)$ can be viewed as a Banach $^*$-algebra when equipped with the norm $\|y\|_1
:= \|y\| + \max\{\|\Dirac_A y\| ,\|y \Dirac_A\|\}$. Thus $\mathcal{C}(H_A)$ plays the role of the `differentiable' elements with respect to this choice of spectral triple.
\end{remark}

\begin{proposition}\label{smoothness}
Let $(\Alg, H_A, \Dirac_A)$ be a spectral triple on $A$ and let $P \in B(H_A)$ be an orthogonal projection commuting with $\Dirac_A$. The three following conditions are equivalent:
\begin{enumerate}
\item $[P, \pi_A(a)] \in \mathcal{C}(H_A)$ for each $a \in \Alg$,
\item $[\Dirac_A^p, \pi_A(a)]$ and $[\Dirac_A^q, \pi_A(a)]$ extend to bounded operators in $B(H_A)$ for each $a \in \Alg$,
\item $[(2P-1)\Dirac_A, \pi_A(a)]$ extends to a bounded operator in $B(H_A)$ for each $a \in \Alg$.
\end{enumerate}
\end{proposition}
\begin{proof}
If (1) holds, then the operators
\begin{eqnarray*}
&& [\Dirac_A^p, \pi_A(a)] = P[\Dirac_A, \pi_A(a)] + [P,\pi_A(a)]\Dirac_A, \\ && [\Dirac_A^q, \pi_A(a)] =
(1-P)[\Dirac_A, \pi_A(a)] - [P,\pi_A(a)]\Dirac_A,
\end{eqnarray*}
viewed as operators on $\textup{dom}(\Dirac_A)$, are bounded for each $a \in \Alg$. We want to regard each of these operators as bounded operators in $B(H_A)$. To this end, we remark that the operator $[\Dirac_A^p, \pi_A(a)]$ is closable, since $P[\Dirac_A, \pi_A(a)]$ and $[P,\pi_A(a)]\Dirac_A$ are closable. Writing cl$(T)$ to denote the closure of $T$, and remarking that $\textup{cl}(P[\Dirac_A, \pi_A(a)]) = P \textup{cl}([\Dirac_A, \pi_A(a)])$, we conclude that
\begin{eqnarray}
&& \textup{cl}([\Dirac_A^p,
\pi_A(a)]) = P \textup{cl}([\Dirac_A, \pi_A(a)]) + \textup{cl}([P,\pi_A(a)]\Dirac_A) \\ && \textup{cl}([\Dirac_A^q, \pi_A(a)]) = (1-P)
\textup{cl}([\Dirac_A, \pi_A(a)]) - \textup{cl}([P,\pi_A(a)]\Dirac_A)
\\ &&  \textup{cl}([{\Dirac_A}, \pi_A(a)]) =  \textup{cl}([\Dirac_A^p, \pi_A(a)]) +  \textup{cl}([\Dirac_A^q, \pi_A(a)]), \label{third}
\end{eqnarray}
where the third identity follows from the first two. Thus (2) holds.

That (2) implies (3) is immediate from the equation $[(2P - 1)\Dirac_A, \pi_A(a)] = [\Dirac_A^p, \pi_A(a)] - [\Dirac_A^q, \pi_A(a)]$.

Finally, if (3) holds then we recover the identity
\begin{eqnarray*}
[(2P-1)\Dirac_A, \pi_A(a)] = (2P-1)[\Dirac_A, \pi_A(a)] + 2[P,\pi_A(a)]\Dirac_A
\end{eqnarray*}
for each $a \in \Alg$. Arguments similar to the first part of the proof now show that the operator $[P,\pi_A(a)]\Dirac_A$ is closable and extends to a bounded operator in $B(H_A)$, so that (1) holds.
\end{proof}
Let us add the following comments on the three equivalent conditions in Proposition \ref{smoothness}.
The first condition  can be compared to a smoothness criterion proposed by Wang in \cite{Wan} whilst the second was studied by Christensen and Ivan in \cite{CI1}. In the special situation in which $P$ is the orthogonal projection onto the span of the positive eigenspaces of $\Dirac_A$ the third condition can be rephrased as requiring that the commutator $[|\Dirac_A|, a]$ is bounded for all $a \in \Alg$. This condition is the first part of Connes and Moscovici's regularity (called smoothness by some authors) which requires for all $a \in \Alg$ that $[\Dirac_A, a]$ and $\delta(a):=[|\Dirac_A|, a]$ are bounded but, moreover,  that $a$ and $[\Dirac_A, a]$ both lie in $\bigcap_{n=1}^{\infty} \textup{dom}(\delta^n)$ (\cite{CMh2}).

The first condition of Proposition \ref{smoothness} has already been mentioned in the introduction as a smoothness assumption. We formalise this in the following definition.

\begin{definition}\label{Toeplitz-type}
Let $(\Alg, H_A, \Dirac_A)$ be a spectral triple on $A$ and let $(\pi_A , \pi_B , P)$ be a Toeplitz triple such that $P$ commutes with $\Dirac_A$. The spectral triple $(\Alg, H_A, \Dirac_A)$ is said to be \textit{$P$-regular} if the equivalent conditions (1), (2) and (3) of Proposition \ref{smoothness} hold.  (Recall (Def.\ref{Toeplitz type}) that in this case we say that the quadruple $(\Alg, H_A, \Dirac_A, P)$ is of Toeplitz type.)
\end{definition}

\subsection{The spectral triple on $E$.}\label{spectriple on extension}

In this section we define  the Dirac operator on the extension algebra and establish some of its basic properties. We suppose that we have C$^*$-algebras $A$, $B$, $E$, a short exact
sequence (\ref{ExtStableIdeal}), spectral triples $(\Alg, H_A, \Dirac_A)$ on $A$ and $(\Balg, H_B, \Dirac_B)$ on
$B$ represented via $\pi_A$ and $\pi_B$ respectively and an orthogonal projection $P \in B(H_A)$ such that
$(\pi_A, \pi_B, P)$ is a Toeplitz triple and the quadruple $(\Alg,H_A,\Dirac_A, P)$ is of Toeplitz type (cf. Defs.\ref{Toeplitz type} and \ref{Toeplitz-type}).

As pointed out in the introduction, the Dirac operator for the extension algebra $E$ will be a combination of two formulae for Kasparov products. 
Recall the definition of the following representations 
$$\Pi_1, \Pi_2 : E \longrightarrow B(H_A \otimes H_B \otimes \C^2)$$  given by
\begin{eqnarray}
\Pi_1 = \pi_{\sigma} \oplus \pi_{\sigma} \textup{ and } \Pi_2 = \pi \oplus \pi_{\sigma}
\end{eqnarray}
and consider unbounded operators $\Dirac_{1}, \Dirac_{2}, \Dirac_{3}$ on $H_A \otimes H_B \otimes \C^2$ given by
\begin{eqnarray}
\Dirac_{1} = 
\begin{bmatrix} \Dirac_A \otimes 1  & 1 \otimes \Dirac_B  \\ 
1 \otimes \Dirac_B  & -\Dirac_A \otimes 1 
\end{bmatrix},
\label{Dirac1}
\end{eqnarray}
\begin{eqnarray}
\Dirac_{2} = 
\begin{bmatrix} \Dirac_A^q \otimes 1 & \Dirac_A^p \otimes 1 \\ 
\Dirac_A^p \otimes 1 & -\Dirac_A^q \otimes 1
\end{bmatrix}
= 
\begin{bmatrix} \Dirac_A^q  & \Dirac_A^p  \\ 
\Dirac_A^p  & -\Dirac_A^q 
\end{bmatrix}
\otimes I
=: \bar{\Dirac}_2 \otimes I
\label{Dirac2}
\end{eqnarray}
and 
\begin{eqnarray}
\Dirac_{3} = 
\begin{bmatrix}  1 \otimes \Dirac_B & 0 \\ 
0 & 1 \otimes \Dirac_B
\end{bmatrix}
= 
I \otimes
\begin{bmatrix} \Dirac_B  & 0  \\ 
0  & \Dirac_B
\end{bmatrix} = : I \otimes \bar{\Dirac}_3,
\label{Dirac3}
\end{eqnarray}
and finally
\begin{eqnarray}
\Dirac_I:=
\begin{bmatrix} 0  & \Dirac_2 - i \Dirac_3  \\ 
\Dirac_2 + i \Dirac_3  & 0
\end{bmatrix},
\label{DiracI}
\end{eqnarray}
an unbounded operator on $H_A \otimes H_B  \otimes \C^4$. 
To begin with, the $\Dirac_i$ are defined on $\mathbb{D}:= \textup{dom}(\Dirac_A) \odot \textup{dom}(\Dirac_B) \otimes \C^2$ and $\Dirac_I$ on $\mathbb{D} \oplus \mathbb{D}$. 
To show that all these operators are essentially self-adjoint we only need to show that each of them possesses a complete orthonormal basis of eigenvectors. Indeed note that if $T: H \to H$ is an unbounded linear operator on a complex separable Hilbert space $H$ with a complete set of orthonormal eigenvectors $(\xi_n) \subseteq H$ and corresponding sequence of real eigenvalues $(\lambda_n)$ so that $T\xi_n = \lambda_n \xi_n$ and thus $\textup{lin} \{\xi_n : n\in \N\} \subseteq \textup{dom}(T)$, then it is easy to see that $T \subseteq T^* = T^{**}=\textup{cl}(T)$ so that $T$ is essentially self-adjoint. Such an operator will be called diagonalisable (with real eigenvalues).  

Not all of the operators defined in (\ref{Dirac1}) - (\ref{DiracI}) have compact resolvent but $\Dirac_1$ and $\Dirac_I$ do have which can be shown by finding their eigenvalues. This also allows to prove summability results.

\begin{lemma}
\label{ExtResolvent} Let $\Dirac_i$, $i=1,2,3$ and $\Dirac_I$ be as above. Then
\begin{enumerate}
\item
$\Dirac_i: \mathbb{D} \to H_A \otimes H_B \otimes \C^2$, $i=1,2,3$ and $\Dirac_I : \mathbb{D} \oplus \mathbb{D}  \to (H_A \otimes H_B \otimes \C^2)^2$ 
are essentially self-adjoint.
\item$\Dirac_1$ and $\Dirac_I$ have compact resolvent (that is, $(I+\Dirac_1^2)^{-1/2} \in \mathcal{K}(H_A \otimes H_B \otimes \C^2)$ and $(I+\Dirac_I^2)^{-1/2} \in \mathcal{K}((H_A \otimes H_B \otimes \C^2)^2)$).
\item 
If $\Dirac_A$ and $\Dirac_B$ are finitely summable then so are $\Dirac_1$ and $\Dirac_I$. Specifically, if $\textup{Tr} (I + \Dirac_A^2)^{-r/2} < \infty$ and $\textup{Tr}(I + \Dirac_B^2)^{-s/2} < \infty$ then  $\textup{Tr}(I + \Dirac_1^2)^{-(r+s)/2} < \infty$ and  $\textup{Tr}(I + \Dirac_I^2)^{-(r+s)/2} < \infty$. Hence if   $(\Alg,H_A,\Dirac_A)$ and
$(\Balg,H_B,\Dirac_B)$ are respectively $r$-summable and $s$-summable then both $\Dirac_1$ and $\Dirac_I$ are
$(r+s)$-summable.
\end{enumerate}
\end{lemma}
\begin{proof}
(1) Note that if $\lambda ,\mu \in \R $ then the eigenvalues of the self-adjoint matrices
$$
\begin{bmatrix} \lambda & \mu \\[1ex] \mu & -\lambda \end{bmatrix}
$$
are $\pm \sqrt{\lambda^2+\mu^2}$, whereas $
\begin{bmatrix} \lambda &  0\\[1ex] 0 & -\lambda \end{bmatrix}
$ and 
$
\begin{bmatrix} 0 & \lambda \\[1ex] \lambda & 0 \end{bmatrix}
$
both have eigenvalues $\pm \lambda$. Now $\Dirac_A$ and $\Dirac_B$ are diagonalizable with orthonormal bases of eigenvectors $(\xi_m) \subseteq H_A$ of $\Dirac_A$ and $(\eta_n)\subseteq H_B$ of $\Dirac_B$ such that there exists $S \subseteq \N$ with $(\xi_m)_{m \in S}$ is an orthonormal basis of $PH_A$, with real eigenvalue sequences $(\lambda_m)$ and $(\mu_n)$ (i.e. $\Dirac_A \xi_m = \lambda_m \xi_m$ and $\Dirac_B \eta_n = \mu_n \eta_n$) and  $|\lambda_m | \to \infty$ and $|\mu_n| \to \infty$. For fixed $m_0$ the subspace $V_{m_0} = \C \xi_{m_0} \otimes \overline{\textup{lin}}(\eta_n) \otimes \C^2$ is $\Dirac_1$-invariant and $\Dirac_1|_{V_{m_0}}$ is given by the matrix
$$
\begin{bmatrix} \lambda_{m_0} I & \textup{Diag}(\mu_n) \\[1ex] \textup{Diag}(\mu_n) & -\lambda_{m_0} I \end{bmatrix}\cong \bigoplus_n \begin{bmatrix} \lambda_{m_0} & \mu_n \\[1ex] \mu_n & -\lambda_{m_0} \end{bmatrix},
$$
and those matrices have eigenvalues $\pm  \sqrt{\lambda_{m_0}^2+\mu_n^2}$. It follows that $\Dirac_1$ is diagonalisable with eigenvalues $\pm  \sqrt{\lambda_m^2+\mu_n^2}$.

$\bar{\Dirac}_2$ has eigenvalues $\pm\lambda_m$, so its eigenvalue sequence $(\lambda_m')$ is given by $\lambda_1,-\lambda_1,\lambda_2, -\lambda_2, \ldots$ with corresponding eigenvectors $e_1,e_2,e_3,\ldots$.  

$\bar{\Dirac}_3$ has the same eigenvalues as $\Dirac_B$ with doubled multiplicity so its eigenvalue sequence $(\mu_n')$ is $\mu_1,\mu_1,\mu_2,\mu_2, \ldots$ with corresponding eigenvectors $f_1,f_2, f_3, \ldots$. 

$\Dirac_I$ restricted to the subspace $V_{m,n} = \C(e_m \otimes f_n) \otimes \C^2$ has the matrix representation
$$
\begin{bmatrix} 0 & \lambda_m' - i\mu_n' \\[1ex]\lambda_m' + i\mu_n' & 0 \end{bmatrix},
$$
and this matrix is diagonalisable with eigenvalues $\pm \sqrt{\lambda_m'^2 + \mu_n'^2}$. 
Therefore all operators $\Dirac_i$, $i=1,2,3$ and $\Dirac_I$ are unbounded and essentially self-adjoint.

\medskip

\noindent
(2) By assumption $|\lambda_m|, |\mu_n| \to \infty$ and therefore also $|\lambda_m'|, |\mu_n'| \to \infty$ as $m,n \to \infty$. Since we have shown in the proof of (1) that the eigenvalues of $\Dirac_1$ and $\Dirac_I$ are given by $\pm  \sqrt{\lambda_m^2 + \mu_n^2}$ and $\pm  \sqrt{\lambda_m'^2 + \mu_n'^2}$ respectively it is easy to see that the sequences of absolute values  of them tend to infinity. Thus $\Dirac_1$ and $\Dirac_I$ have compact resolvent.

\medskip

\noindent
(3) Finally assuming $\textup{Tr} (I + \Dirac_A^2)^{-r/2} < \infty$ and $\textup{Tr}(I + \Dirac_B^2)^{-s/2} < \infty$ means $\sum_m (1 + \lambda_m^2)^{-r/2} < \infty$ and 
 $\sum_n (1 + \mu_n^2)^{-s/2} < \infty$. As indicated in \cite{HSWZ} in a similar context, the inequality
$$(x+y-1)^{-(\alpha+\beta)} \leq x^{-\alpha}y^{-\beta},$$ valid for $x,y >1$ and $\alpha,\beta>0$ then implies that $\sum_{m ,n}(1 + \lambda_m^2+ \mu_n^2)^{-(r+s)/2} < \infty$.
This shows that $\Dirac_1$ is $(r+s)$-summable.  Since the eigenvalues of $\Dirac_I$ are given by $\pm  \sqrt{\lambda_m'^2 + \mu_n'^2}$ and $(\lambda_m'^2)$ and $(\mu_n'^2)$ are just the sequences $(\lambda_m^2)$ and $(\mu_n^2)$ with each term repeated once it follows that also $\Dirac_I$ is $(r+s)$-summable.
\end{proof}

Next we need to show boundedness of commutators with our operators $\Dirac_1$ and $\Dirac_I$. In order to  do so let us point out the following elementary identity for commutators of matrices. 
\begin{eqnarray}
\left[ 
\begin{bmatrix} a_{11} & a_{12} \\[1ex] a_{21} & a_{22} \end{bmatrix}, \begin{bmatrix} a& 0 \\[1ex] 0 & b \end{bmatrix} \right] = \begin{bmatrix} [a_{11},a] & a_{12}b - a a_{12} \\[1ex] a_{21}a - ba_{21} & [a_{22},b] \end{bmatrix}.
\label{diagcom}
\end{eqnarray}

\begin{lemma} \label{ExtCommutators}
Let  $e$ be in the dense $^*$-subalgebra $\mathcal{E}$ of $E$ generated by elementary tensors $k \otimes b \in \mathcal{C}(P H_A) \odot
\Balg$ and $\{PaP \otimes I_B: \;a \in
\Alg\}$. Then the operators $[\Dirac_1,\Pi_1(e)]$
%:  \mathbb{D} \to  H_A \otimes H_B \otimes \C^2$ 
and  
$[\Dirac_I,\Pi_2(e) \oplus \Pi_2(e)]$
%:  \mathbb{D} \oplus \mathbb{D} \to (H_A \otimes H_B \otimes \C^2)^2$
are bounded.
\end{lemma}
\begin{proof}
Let $e= x + P\pi_A(a)P\otimes 1$ be an element in $\mathcal{E}$, where $x \in \mathcal{C}(P H_A) \odot \Balg$. Then direct calculations using (\ref{diagcom}) reveal that
\begin{eqnarray*}
[\Dirac_1, \Pi_1(e)] = \begin{bmatrix} [\Dirac_A, \pi_A(a)] \otimes 1 & 0 \\[1ex] 0 & -[\Dirac_A, \pi_A(a)]
\otimes 1 \end{bmatrix},
\end{eqnarray*}
which is bounded. Next we determine the following commutators, again using (\ref{diagcom}). (We will omit the $A$ and $B$ indices of $\Dirac_A$ and $\Dirac_B$ as well as the representation $\pi_A$.)
\begin{eqnarray*}
[\Dirac_2, \Pi_2(e)] 
&=& 
\left[ 
\begin{bmatrix} \Dirac^q \otimes 1 & \Dirac^p \otimes 1 \\ 
\Dirac^p \otimes 1  &
-\Dirac^q \otimes 1
\end{bmatrix}, \begin{bmatrix} 
\pi(e) &  0 \\ 
0 &  \pi_{\sigma}(e) 
\end{bmatrix} 
\right] \\
&=&
\begin{bmatrix} [\Dirac^q \otimes 1, \pi(e) ] & (\Dirac^p \otimes 1 ) \pi_{\sigma} (e) - \pi(e) (\Dirac^p \otimes 1)  \\ 
(\Dirac^p \otimes 1) \pi (e) - \pi_{\sigma} (e) (\Dirac^p \otimes 1)  &
-[ \Dirac^q \otimes 1 , \pi_{\sigma} (e) ]
\end{bmatrix} \\
&=&
\begin{bmatrix} 0 & \Dirac^p a\otimes 1 - PaP\Dirac^p \otimes 1 - x(\Dirac^p \otimes 1)  \\ 
(\Dirac^pPaP \otimes 1  + (\Dirac^p \otimes1)x - a\Dirac^p \otimes 1  &
-[ \Dirac^q,a] \otimes 1 
\end{bmatrix} \\
&=&
\begin{bmatrix} 0 & P[\Dirac^p, a] \otimes 1 - x(\Dirac^p \otimes 1)  \\ 
[\Dirac^p,a] P \otimes 1  + (\Dirac^p \otimes1)x &
-[ \Dirac^q,a] \otimes 1 
\end{bmatrix}, \\
\end{eqnarray*}
where we have used that $\Dirac^q \otimes 1 \perp \pi(e)$. Next
\begin{eqnarray*}
[ \Dirac_3 , \Pi_2 (e) ]
&=&
\left[ 
\begin{bmatrix} 1 \otimes \Dirac_B & 0 \\ 
0  &
1 \otimes \Dirac_B
\end{bmatrix},
\begin{bmatrix} 
\pi(e) &  0 \\ 
0 &  \pi_{\sigma}(e) 
\end{bmatrix} 
\right] \\
&=&
\begin{bmatrix} [1 \otimes \Dirac_B, \pi(e) ] & 0 \\ 
 0  &
[1 \otimes \Dirac_B , \pi_{\sigma} (e) ]
\end{bmatrix} \\
&=&
\begin{bmatrix} [1 \otimes \Dirac_B, x ] & 0 \\ 
 0  & 0
\end{bmatrix}. \\
\end{eqnarray*}
Using these identities we obtain
\begin{eqnarray*}
[ \Dirac_I , \Pi_2 (e) \oplus \Pi_2(e) ]
&=&
\left[ 
\begin{bmatrix} 0  & \Dirac_2 - i \Dirac_3  \\ 
\Dirac_2 + i \Dirac_3  & 0
\end{bmatrix},
\begin{bmatrix} 
\Pi_2(e) &  0 \\ 
0 &  \Pi_2(e) 
\end{bmatrix} 
\right] \\
&=&
\begin{bmatrix} 0  & [\Dirac_2 - i \Dirac_3, \Pi_2(e)]  \\ 
[\Dirac_2 + i \Dirac_3,\Pi_2(e)]  & 0
\end{bmatrix} 
\end{eqnarray*}
and
\begin{eqnarray*}
 [\Dirac_2 - i \Dirac_3, \Pi_2(e)]
&=&
\begin{bmatrix} - i [1 \otimes \Dirac,x] &   P[\Dirac^p, a] \otimes 1 - x(\Dirac^p \otimes 1)  \\ 
 [\Dirac^p,a] P \otimes 1  + (\Dirac^p \otimes1)x &
-[ \Dirac^q,a] \otimes 1    
\end{bmatrix} 
\end{eqnarray*}
\begin{eqnarray*}
[\Dirac_2 + i \Dirac_3, \Pi_2(e)]
&=&
\begin{bmatrix}
 i [1 \otimes \Dirac,x] &   P[\Dirac^p, a] \otimes 1 - x(\Dirac^p \otimes 1)   \\ 
 [\Dirac^p,a] P \otimes 1  + (\Dirac^p \otimes1)x & -[ \Dirac^q,a] \otimes 1  
\end{bmatrix} .
\end{eqnarray*}
The claim is now evident since all entries in all operator matrices of the commutators are indeed bounded. 
\end{proof}
We are now ready to state and prove the first of our main results.
\begin{theorem}\label{ConstExt}
Let $A$ and $B$ be unital C$^*$-algebras and suppose that $E$ arises as the short exact sequence
(\ref{extension}) and that there exist spectral triples $(\Alg,H_A,\Dirac_A)$ on $A$ and
$(\Balg,H_B,\Dirac_B)$ on $B$, represented via $\pi_A$ and $\pi_B$ respectively, and an orthogonal projection $P \in B(H_A)$ such that $(\Alg, H_A, \Dirac_A, P)$ is of Toeplitz type. Let 
\begin{eqnarray*}
\Pi = \Pi_1\oplus \Pi_2 \oplus \Pi_{2}, \;\;H = (H_A \otimes
H_B \otimes \C^2)^3, \textup{ and}
\end{eqnarray*}
\begin{eqnarray*}
\Dirac = \begin{bmatrix} \Dirac_{1} & 0 & 0 \\[1ex] 
0 & 0 &   \Dirac_2 - i \Dirac_{3}\\[1ex]  
0  & \Dirac_2  +
i \Dirac_{3} &0 \end{bmatrix}.
\end{eqnarray*}
Then $(\mathcal{E}, H, \Dirac)$, represented
via $\Pi$, defines a spectral triple on $E$. 
Moreover, the spectral dimension of this spectral triple is computed by
the identity
\begin{eqnarray*}
s_0(\mathcal{E}, H, \Dirac) = s_0(\Alg,H_A,\Dirac_A) + s_0(\Balg,H_B,\Dirac_B).
\end{eqnarray*}
\end{theorem}
\begin{proof}
Note first that the representation $\Pi$ is faithful and that $\Dirac = \Dirac_1 \oplus \Dirac_I$. 
Since $\Dirac_1$ and $\Dirac_I$ have compact resolvent by Lemma \ref{ExtResolvent}.(2) so has $\Dirac$. 
By Lemma \ref{ExtCommutators} the commutators $[\Dirac, \Pi (e)]= [\Dirac_1 , \Pi_1 (e)] \oplus [\Dirac_I , \Pi_2 (e) ]$ are indeed bounded for every $ e \in \mathcal{E}$.
The summability claim finally follows 
since $s_0(\Dirac) = s_0 (\Dirac_1 \oplus \Dirac_I) = \max( s_0(\Dirac_1)  , s_0(\Dirac_I)) $ and we have shown in Lemma \ref{ExtResolvent}.(3) that $s_0(\Dirac_1)  = s_0(\Dirac_I)= s_0(\Alg,H_A,\Dirac_A) + s_0(\Balg,H_B,\Dirac_B)$.
\end{proof}

\subsection{The algebra $C^1(E)$}\label{algC1E}
Theorem \ref{ConstExt}  only provides the existence of a spectral triple for the dense subalgebra $\mathcal{E}$ of $E$. Given the Dirac operator we defined, it is natural to ask how large we can allow the dense subalgebra to be. More specifically, we ask: what is the largest `smooth' subalgebra of $E$ in which the construction in Theorem \ref{ConstExt}  defines a spectral triple?
There seems to be a natural such algebra, the maximal Lipschitz algebra $C^1(E)$ associated to our Dirac operator mentioned after Def.\ref{deftriple}. It is also an extension fitting into the short exact sequence (\ref{smoothext}).

We think of $E$ as being represented via $\pi$ on $H_A \otimes H_B$ so that $E \subseteq \mathcal{K}(PH_A) \otimes B + PAP \otimes 1$, where the sum is algebraically direct. 
Then $C^1(E)$ is the $*$-subalgebra of $E$ comprising all $e \in E$ such that
\renewcommand{\theenumi}{\roman{enumi}}
\begin{enumerate}
\item $\Pi(e) \textup{dom}(\Dirac) \subseteq \textup{dom}(\Dirac)$,
\item $[\Dirac, \Pi(e)]: \textup{dom}(\Dirac) \to H$ is closable and  bounded.
\end{enumerate}

Writing $e=x + PaP \otimes I$ uniquely, where $x \in \mathcal{K}(PH_A) \otimes B$ and $a \in A$ the formulas for $[\Dirac_1,\Pi_1(e)]$ and $[\Dirac_I,\Pi_2(e) \oplus \Pi_2(e)]$ in the proof of Lemma \ref{ExtCommutators} show that $e \in C^1(E)$ iff the following conditions are satisfied.
\begin{enumerate}
\item
$\pi_A(a) (\textup{dom}(\Dirac_A)) \subseteq \textup{dom}(\Dirac_A)$ and $[\Dirac_A, \pi_A(a)]$ is closable and bounded.
\item
$\pi (x)   (\textup{dom}(1  \otimes \Dirac_B)) \subseteq \textup{dom}(1 \otimes \Dirac_B)$
and $(1 \otimes \Dirac_B) x , x (1 \otimes \Dirac_B) $ are bounded.
\item
$\pi_A(a) (\textup{dom}(\Dirac_A^p)) \subseteq \textup{dom}(\Dirac_A)$ and $\pi_A(a) (\textup{dom}(\Dirac_A^q)) \subseteq \textup{dom}(\Dirac_A)$ and 
$[\Dirac_A^p, \pi_A(a)] $, $[\Dirac_A^q, \pi_A(a)]$ are closable and bounded. 
\end{enumerate}

We now define $C^1(\mathcal{K}_B) \subseteq C^1(E)$ to be the  dense $^*$-subalgebra of $\mathcal{K}_B$ consisting of all $x \in \mathcal{K}_B$ such that
$\pi (x)   (\textup{dom}(1  \otimes \Dirac_B)) \subseteq \textup{dom}(1 \otimes \Dirac_B)$
and $(1 \otimes \Dirac_B) \pi(x) , \pi(x) (1 \otimes \Dirac_B) $ are bounded. (More easily we could define $C^1(\mathcal{K}_B)= \{ x \in   C^1(E) : x \in \mathcal{K}_B \}$.) 

\vspace{3mm}

Finally, let $C^{1,P}(A)$ be the $^*$-subalgebra of $A$ consisting of all $a \in A$ such that
\begin{enumerate}
\item
$\pi_A(a) (\textup{dom}(\Dirac_A)) \subseteq \textup{dom}(\Dirac_A)$ and $[\Dirac_A, \pi_A(a)]$ is closable and bounded.
\item
$\pi_A(a) (\textup{dom}(\Dirac_A^p)) \subseteq \textup{dom}(\Dirac_A)$ and $\pi_A(a) (\textup{dom}(\Dirac_A^q)) \subseteq \textup{dom}(\Dirac_A)$ and 
$[\Dirac_A^p, \pi_A(a)] $, $[\Dirac_A^q, \pi_A(a)]$ are closable and bounded. 
\end{enumerate}
\vspace{5mm}

\begin{remark}\label{4.8}
Our definitions imply that we obtain the following short exact sequence
\begin{eqnarray}
\xymatrix{ 0 \ar[r] & C^1(\mathcal{K}_B) \ar[r]^{\iota_1} & C^1(E) \ar[r]^{\sigma_1} & C^{1,P}(A) \ar[r] & 0},
\label{smoothext}
\end{eqnarray}
where $\iota_1$ and $\sigma_1$ are the natural inclusion and quotient map respectively.
\end{remark}

\section{The metric condition for extensions.}\label{metcon}

We are interested in the construction of spectral metric spaces and, as such, we turn now to the question of
whether the spectral triple on $E$ satisfies Rieffel's metric condition (cf. Prop.\ref{metric condition}), giving $E$ the structure of a spectral metric space. There is an abundance of Lipschitz seminorms on each of $A$, $B$ and $E$ which we could study, depending on the choice of smooth subalgebras. In this section we will focus on the situation in which the smooth subalgebras (cf. Definition \ref{deftriple}) are $\Alg = C^{1,P}(A)$ and $\Balg = C^1(B)$ and show that it is possible to construct a Lipschitz seminorm on $C^1(E)$ coming from a spectral triple with the desired properties. Our results can be adjusted to fit the setting of dense subalgebras possibly smaller than $C^{1,P}(A)$ or $C^1(B)$.

To this end, we assume that the spectral triples $(C^{1,P}(A),H_A,\Dirac_A)$ and $(C^1(B),H_B,\Dirac_B)$ on $A$ and $B$ respectively, together with Lipschitz seminorms $L_{\Dirac_A}$ on $C^1(A)$ and $L_{\Dirac_B}$ on $C^1(\mathcal{K}_B)$,  give $A$ and $B$ the structure of spectral metric spaces. According to Rieffel's criterium (Proposition \ref{metric condition}), this means that the spectral triples $(\Alg,H_A,\Dirac_A)$ and $(\Balg,H_B,\Dirac_B)$ are
nondegenerate, that the spaces
\begin{eqnarray*}
\widetilde{\mathcal{U}}_A &:=& \{\tilde{a} \in C^{1,P}(A) / \C I_A: \;\|[\Dirac_A, \pi_A(\tilde{a})]\| \leq 1\}, \\
\widetilde{\mathcal{U}}_B &:=& \{\tilde{b} \in C^1(B) / \C I_B: \;\|[\Dirac_B, \pi_B(\tilde{b})]\| \leq 1\}
\end{eqnarray*}
are bounded subsets of $A / \C I_A$ and $B / \C I_B$ respectively and that the sets
\begin{eqnarray*}
\mathcal{U}_{A,1} &:=& \{a \in  C^{1,P}(A): \;\|a\| \leq 1, \;\|[\Dirac_A, \pi_A(a)]\| \leq 1\}, \\
\mathcal{U}_{B,1} &:=& \{b \in C^1(B): \;\|b\| \leq 1, \;\|[\Dirac_B, \pi_B(b)]\| \leq 1\}
\end{eqnarray*}
are norm totally bounded. From Rem.\ref{4.8} and the standing assumption of essentialness of our extension we conclude that for every element $\tilde{e} \in C^1(E)  / \C I_E $ there is a unique $x \in C^1(\mathcal{K}_B)$ and $\tilde{a} \in C^{1,P}(A) / \C I_A$ such that $\tilde{e} = (x + PaP \otimes I)^{\sim}$. In this sense we have: 
\begin{remark}
The equality $C^1(E) / \C I_E = C^1(\mathcal{K}_B) + P (C^{1,P}(A) / \C
I_A) P \otimes \C I_B$ holds, where the sum is direct.
\end{remark}
Notice that $I_E = P I_A P \otimes I_B= P \otimes I_B$. 

\vspace{2mm}

We now introduce the following spaces, where $X$ and $Y$ play the role of subscripts and do not refer to other objects.

\begin{eqnarray*}
\mathcal{U}_X &:=& \{x \in C^1(\mathcal{K}_B): \;\|[\Dirac_I, \Pi_2(x)\oplus \Pi_2(x)]\| \leq 1\}, \\
\widetilde{\mathcal{U}}_Y &:=& \{P\tilde{a}P \otimes I_B \in P (C^{1,P}(A) / \C
I_A) P \otimes \C I_B: \;\|[\Dirac_A, \pi_A(\tilde{a})]\|
\leq 1\}, \\
\mathcal{U}_{Y,1} &:=& \{PaP \otimes I_B \in P C^{1,P}(A)P \otimes \C I_B: \; \| a \| \leq 1,\;\;\|[\Dirac_A, \pi_A(a)]\| \leq
1\}.
\end{eqnarray*}

Recall  from Def.\ref{Toeplitz type} that a Toeplitz type quadruple $(\Alg, H_A, \Dirac_A, P)$ is called $P$-injective if  $\ker ( \Dirac_A^p ) \cap P H_A  = \{0\}$ (i.e. the operator $\Dirac_A^p|_{P H_A}$ has an inverse in
    $\mathcal{K}(PH_A)$.)
It will be necessary for us to impose this condition throughout this section.

We remark that necessarily $\ker(\Dirac_A^p \cap P H_A)$ is finite rank, so that if $P$-injectivity 
fails then we can merely replace $P$ with $P - P_{\ker (\Dirac_A)}$, where $ P_{\ker (\Dirac_A)}$ is the orthogonal projection onto $ \ker (\Dirac_A)$. This procedure does not affect any
other aspects of the extension. Using Lemma \ref{ExtResolvent} we have the following observation.
\begin{remark}\label{commutator}
One checks that the expressions for $[\Dirac_1, \Pi_1(e)]$ and $[\Dirac_I, \Pi_2(e) \oplus \Pi_2(e)]$ in the proof of Lemma \ref{ExtCommutators} define bounded operators for all $e \in C^1(E) \supseteq \mathcal{E}$.
Indeed, for any $e = x + PaP \otimes I_B \in C^1(E)$ such that $x \in C^1(\mathcal{K}_B)$ and $a \in C^{1,P}(A)$ we have
\begin{eqnarray}
[\Dirac_1, \Pi_1(e)] = \begin{bmatrix} [\Dirac_A, \pi_A(a)] \otimes 1 & 0 \\[1ex] 0 & -[\Dirac_A, \pi_A(a)]
\otimes 1 \end{bmatrix},
\label{1com}
\end{eqnarray}
\begin{eqnarray}
[ \Dirac_I , \Pi_2 (e) \oplus \Pi_2(e) ]
=
\begin{bmatrix} 0  & [\Dirac_2 - i \Dirac_3, \Pi_2(e)]  \\ 
[\Dirac_2 + i \Dirac_3,\Pi_2(e)]  & 0
\end{bmatrix}
\label{2com} 
\end{eqnarray}
\begin{eqnarray}
[\Dirac_2 \pm i \Dirac_3, \Pi_2(e)]
=
\begin{bmatrix}
 \pm i [1 \otimes \Dirac,x] &   P[\Dirac^p, a] \otimes 1 - x(\Dirac^p \otimes 1)   \\ 
 [\Dirac^p,a] P \otimes 1  + (\Dirac^p \otimes1)x & -[ \Dirac^q,a] \otimes 1  
\end{bmatrix}. 
\label{3com}
\end{eqnarray}

\end{remark}

The spectral triple $(\mathcal{E}, \Pi, \Dirac)$ on $E$ in Theorem \ref{ConstExt} determines a seminorm $L= L_{\Dirac}$
on $C^1(E)$ given by $L(e) = \max\{\|[\Dirac_1, \Pi_1(e)]\|,\|[\Dirac_I, \Pi_2(e) \oplus \Pi_2(e) ]\|\}$. Our first objective is to show that $L$ is a nondegenerate Lipschitz seminorm:

\begin{proposition}\label{Lip-norm}
Let $e \in C^1(E)$, then  $L(e) = 0$ iff there exists $\lambda \in \C$ with $e= \lambda I$.
\end{proposition}
\begin{proof}
The proof consists in showing $C^1(E) \cap L^{-1}(\{0\}) = \C I_E$. To this end, let $e = x + PaP \otimes I_B$, where $x \in C^1(\mathcal{K}_B)$ and $a \in C^{1,P}(A)$. If $L(e) = 0$ then  $[\Dirac_1, \Pi_1(e)] = 0$ so that $[\Dirac_A, \pi_A(a)] = 0$ from eq. (\ref{1com}). Since $\Dirac_A$ implements a nondegenerate spectral triple on $A$, necessarily $a = \lambda I_A$ for some $\lambda \in \C$, so that we can write $e = x + \lambda I_E$.
Moreover, we have $[\Dirac_I, \Pi_2(x) \oplus \Pi_2(x)] = 0$, so by (\ref{2com}) and (\ref{3com})  this means that $(\Dirac_A^p \otimes 1)\pi(x)= (\Dirac_A^p \otimes 1)x = 0$. By $P$-injectivity, $x = 0$, completing the proof.
\end{proof}

\begin{lemma}\label{UE-Lip}
Let $\widetilde{\mathcal{U}}_E$ and $\mathcal{U}_{E,1}$ be the subsets of $E / \C I_E$ and $E$ respectively defined by
\begin{eqnarray*}
\widetilde{\mathcal{U}}_E &:=& \{\tilde{e} \in C^1(E) / \C I_E:\;L(\tilde{e}) \leq 1\},\\
\mathcal{U}_{E,1} &:=& \{e \in C^1(E): \|e\| \leq 1, \;L(e) \leq 1\}.
\end{eqnarray*}
Then 
$$
\widetilde{\mathcal{U}}_E \subseteq 7\mathcal{U}_X + \widetilde{\mathcal{U}}_Y:= \{ x+ PaP \otimes I/\C P \otimes I : \| [\Dirac_I, \Pi_2(x) \oplus \Pi_2(x)] \| \leq 7, \| [\Dirac_A , \pi_A(a) ] \| \leq 1 \}
$$ 
and 
$$\mathcal{U}_{E,1} \subseteq
7\mathcal{U}_{X} + \mathcal{U}_{Y,1}=  \{ x+ PaP \otimes \C I : \| [\Dirac_I, \Pi_2(x) \oplus \Pi_2(x) ] \| \leq 7, \| [\Dirac_A , \pi_A(a) ] \| \leq 1 \}.$$
\end{lemma}
\begin{proof} To show the first inclusion, 
let $e \in C^1(E)$ be such that $L(e) \leq 1$ and write $e = x + PaP \otimes I_B$ for unique $x \in \mathcal{K}_B$ and $a \in C^{1,P}(A)$. We need to show that $\| [ \Dirac_A , \pi_A(a)]\| \leq 1$ and $\| [\Dirac_I , \Pi_2(x) \oplus \Pi_2(x) ]\| \leq 7$. 

Now $L(e) \leq 1$ is equivalent to $\|[\Dirac_1, \Pi_1(e)]\| \leq 1$ and $\|[\Dirac_I, \Pi_2(e) \oplus \Pi_2(e)]\| \leq 1$. This means that the norms of all entries in the matrix expressions in Rem.\ref{commutator} are bounded by 1 from which we obtain the following inequalities: 
\begin{enumerate}
\item
$\|[\Dirac_A, \pi_A(a)]\|= \|[\Dirac_A , a] \| \leq 1 $,
\item
$\| [ 1 \otimes \Dirac_B , x ]\| \leq 1$, 
\item
$\| P [ \Dirac_A^p , a] \otimes 1 + x (D_A^p \otimes 1) \| \leq 1  $,
\item
$\| [D_A^p ,a] P \otimes 1 + (D^p \otimes 1)x \| \leq 1$,
\item
$\| [D_A^q , a ] \| \leq 1$.
\end{enumerate}
Since $[ \Dirac_A , a ]= [ \Dirac_A^p , a]+[ \Dirac_A^q , a]$, we must have $\|[\Dirac_A^p,
a]\| \leq 2$, using (i) and (v).
Then (iii) and (iv) imply 
\begin{eqnarray}
\| (\Dirac^p \otimes 1)x \|, \|x(\Dirac^p \otimes 1)\| \leq 3. 
\label{xdp}
\end{eqnarray}
Now (ii) and (\ref{xdp}) imply
\begin{eqnarray*}
\|[\Dirac_2 \pm i \Dirac_3, \Pi_2(x)]\|
&=&
\left\|
\begin{bmatrix}
 \pm i [1 \otimes \Dirac,x] &   - x(\Dirac^p \otimes 1)   \\ 
  (\Dirac^p \otimes1)x & 0  
\end{bmatrix} 
\right\| \\
&\leq &
\| [1 \otimes \Dirac_B , x] \| + \| x (\Dirac^p \otimes 1 ) \| + \| (\Dirac^p \otimes 1) x \|  \\
&\leq &
7
\end{eqnarray*}
This shows that $x \in 7 \: \mathcal{U}_X$ and the result follows. 
The second inclusion can be shown in a similar way.
\end{proof}

The next Lemma is immediate from our definitions.

\begin{lemma}\label{isomet}
Let $\sigma: E \to A$ be the quotient map with induced map $\tilde{\sigma} : E / \C I \to A / \C I$. Then the maps
\begin{eqnarray*}
\sigma|_{\mathcal{U}_{Y,1}}: (\mathcal{U}_{Y,1}, \|\cdot\|_E) \rightarrow (\mathcal{U}_{A,1}, \|\cdot\|_A), \;\;\;\;\tilde{\sigma}|_{\widetilde{\mathcal{U}}_Y}:
(\widetilde{\mathcal{U}}_Y, \|\cdot\|_{E / \C I_E}) \rightarrow (\widetilde{\mathcal{U}}_A, \|\cdot\|_{A / \C I_A})
\end{eqnarray*}
are isometric bijections. Therefore, since $(C^{1,P}(A), H_A , \Dirac_A)$ satisfies Rieffel's metric condition, $\widetilde{\mathcal{U}}_Y \subseteq E / \C I_E$ is bounded and $\mathcal{U}_{Y,1} \subseteq E$ is totally bounded.
\end{lemma}

The uniform norm estimates in the next result is of key importance to establish Rieffel's metric condition. It uses a norm estimate from the proof of Lemma \ref{UE-Lip}.

\begin{lemma}\label{approx}
Let $Y := ({\Dirac_A^p|_{P H_A}})^{-1} \in \mathcal{K}(PH_A)$, let
$\{P_k\}_{k \in \N}$ be the spectral projections of $Y$ and write $Q_n := \sum_{k =
1}^n P_k$. Then for each $\epsilon > 0$ there exists an $N \in \N$ such that, for each $ x \in \mathcal{U}_X$ and for all $n \geq N$,
\begin{eqnarray*}
\|x - (Q_n \otimes 1) x (Q_n \otimes 1) \| \leq \epsilon.
\end{eqnarray*}
Moreover, for each $x \in \mathcal{U}_X$ and for each $n \in \N$, $\|x_n\|\leq 3 \|Y\|$, where $x_n := (Q_n \otimes 1) x (Q_n \otimes 1)$.
\end{lemma}
\begin{proof} Since $Y$ is a compact operator, it quickly follows that for each
$\epsilon > 0$ there exists an $N \in \N$ such that $\|Y - YQ_n\| \leq \frac{\epsilon}{6}$ and $\|Y - Q_n Y\| \leq
\frac{\epsilon}{6}$ whenever $n \geq N$. For $x \in \mathcal{U}_X$, using $P \Dirac_A Y = P$, we obtain
\begin{eqnarray*}
\|(Q_n \otimes 1) x (Q_n \otimes 1)\| 
&\leq& 
\|(Q_n \otimes 1) x (P \Dirac_A Y \otimes 1 ) (Q_n \otimes 1)\| \\
&\leq& 
\|(Q_n \otimes 1) x (P \Dirac_A \otimes 1 )( Y \otimes 1) (Q_n \otimes 1)\| \\
&\leq& 
\|Q_n \| \|  x (P \Dirac_A \otimes 1 ) \| \| Y \| \|Q_n \| \\
& = &
\|x(P \Dirac_A \otimes 1)\| \|Y\| \leq 3\|Y\|,
\end{eqnarray*}
where the last inequality follows from (\ref{xdp}) in the proof of Lemma \ref{UE-Lip}.
This proves the second statement. To prove the first statement note that for all $x \in \mathcal{U}_X$ and $ n \geq N$ 
\begin{eqnarray*}
\|x - x (Q_n \otimes 1)\| & \leq & \|x (P \Dirac_A \otimes 1)( Y \otimes 1) - x (P \Dirac_A \otimes 1 ) (Y Q_n \otimes 1)\| \\
 & \leq & \|x(P \Dirac_A \otimes 1)\| \|Y \otimes 1 - Y Q_n \otimes 1 \| \\
 & = & \|x(P \Dirac_A \otimes 1)\| \|Y - Y Q_n \| \\
 & \leq & \frac{\epsilon}{2},
\end{eqnarray*}
and similarly $\|x - (Q_n \otimes 1) x\| \leq \frac{\epsilon}{2}$, so that
\begin{eqnarray*}
\|x - (Q_n \otimes 1) x (Q_n \otimes 1) x\| 
&\leq& 
\|x - x (Q_n \otimes 1)\| + \|x (Q_n \otimes 1) - (Q_n \otimes 1) x (Q_n \otimes 1)\| \\
&\leq & \epsilon.
\end{eqnarray*}

\end{proof}

We can now prove our second main result.

\begin{theorem}\label{ConstExt2}
Let $A$ and $B$ be unital C$^*$-algebras and suppose $E$ arises as the short exact sequence
(\ref{extension}). Suppose further that there exists spectral triples $(\Alg,H_A,\Dirac_A)$ on $A$ and
$(\Balg,H_B,\Dirac_B)$ on $B$, represented via $\pi_A$ and $\pi_B$ respectively, and an orthogonal projection $P \in B(H_A)$ such that 
$(\Alg, H_A, \Dirac_A, P)$ is of Toeplitz type and $P$-injective. If the spectral triples $(\Alg,H_A,\Dirac_A)$ and
$(\Balg,H_B,\Dirac_B)$ satisfy Rieffel's metric condition then so does the spectral triple $(\mathcal{E}, H, \Dirac)$ so that $(E, L_{\Dirac})$ is a spectral metric space.
\end{theorem}
\begin{proof}
According to Rieffel's criteria (Prop.\ref{metric condition}) we need to show that $\widetilde{\mathcal{U}}_E$ is bounded and $\mathcal{U}_{E,1}$ is totally bounded.
By Lemma \ref{isomet} we know that $\widetilde{\mathcal{U}}_Y$ and $ \mathcal{U}_{Y,1}$ are bounded, respectively totally bounded. By Lemma \ref{UE-Lip} we know that $\widetilde{\mathcal{U}}_E \subseteq 7\mathcal{U}_X + \widetilde{\mathcal{U}}_Y$ and $\mathcal{U}_{E,1} \subseteq 7\mathcal{U}_{X} + \mathcal{U}_{Y,1}$. So we have only to show that the set $\mathcal{U}_{X} \subseteq \mathcal{K}_B$ is  totally norm bounded. Using Lemma \ref{approx}, it will suffice for us to show that the sets 
$$
(Q_n \otimes 1) \mathcal{U}_X (Q_n \otimes 1)
$$ 
are totally bounded for each $n \in \N$. Since we may regard $(Q_n \otimes 1) \mathcal{U}_X (Q_n \otimes 1)$ as a subset of $M_{m_n}(B)$, where $m_n = \textup{dim}(Q_n)$, any given element in this set can be expressed in the form
\begin{eqnarray*}
x_n = \sum_{i,j = 1}^{m_n} \pi_B(b_{i,j}) \otimes (| e_j \rangle \langle e_i|) ,
\end{eqnarray*}
where $b_{i,j} \in B$ and $\{e_i\}_{i = 1}^{m_n}$ is an orthonormal basis for the finite dimensional
Hilbert space $Q_n H_A$. We shall denote the corresponding projections in $B(Q_n H_A \otimes H_B)$ by $\{p_i\}_{i =
1}^{m_n}$. Since these commute with
$1 \otimes \Dirac_B$, we have that for $x \in \mathcal{U}_X$ and $n \in \N$,
\begin{eqnarray*}
\|\pi_B(b_{i,j})\| = \|p_j x_n p_i\| \leq \|x_n\| \leq 3 \|Y\|,
\end{eqnarray*}
\begin{eqnarray*}
\|[\Dirac_B, \pi_B(b_{i,j})]\| = \|[1 \otimes \Dirac_B, p_j x_n p_i]\| = \|p_j [1 \otimes \Dirac_B, x] p_i\| \leq
1,
\end{eqnarray*}
since  $\|[1 \otimes \Dirac_B, x] \| \leq 1$ from $\| \Dirac_I , \Pi_2(x) \oplus\Pi_2(x)]\| \leq 1$.  These estimates tell us that the sets $Q_n \mathcal{U}_X Q_n$ are contained in the sets
\begin{eqnarray*}
S_n &:=& \Big\{ \sum_{i,j = 1}^{m_n}
\pi_B(b_{i,j}) \otimes (| e_j \rangle \langle e_i|): \;\;\; b_{i,j} \in \mathcal{B}, \;\|b_{i,j}\| \leq 3\|Y\|, \;\|[\Dirac_B,
\pi_B(b_{i,j})]\| \leq 1 \Big\} \\
& \subs & \Big\{ \sum_{i,j = 1}^{m_n}
\pi_B(b_{i,j}) \otimes (| e_j \rangle \langle e_i|): \;\;\; b_{i,j} \in 3\|Y\| \mathcal{U}_{B,1} \Big\}.
\end{eqnarray*}
Now we recall our assumption that the spectral triple on $B$ satisfies Rieffel's metric condition, so
that $\mathcal{U}_{B,1}$ is totally bounded and consequently the sets $S_n$ are totally bounded as well. This
concludes the proof of the Theorem.
\end{proof}

\section{Examples}\label{eg}

\subsection{Split extensions.} Recall that an extension (\ref{extension}) is \textit{split} when it is semisplit and the
splitting map $s: A \to \mathcal{L_B}$ can be chosen to be a $*$-homomorphism (rather than merely a completely positive map). If such an extension admits a Toeplitz representation, as in Definition \ref{Toeplitz type}, then $P$ is the identity in $B(\Hil_A)$, and we can restrict our attention to representations of this type. This significantly reduces the technicalities associated with the construction of spectral triples on such extensions. Our construction in Theorem \ref{ConstExt} reads in this case as follows:

\begin{proposition}
Let $A$ and $B$ be unital C$^*$-algebras, endowed with spectral triples $(\Alg,H_A,\Dirac_A)$ and
$(\Balg,H_B,\Dirac_B)$ respectively. Let $E \cong \mathcal{K}(H_A) \otimes B + A \otimes I_B$ be a unital split extension of $A$ by the stabilisation of $B$. Then $(\mathcal{E}, H, \Dirac)$, represented via $\Pi$, defines a spectral triple on $E$. Here,
\begin{eqnarray*}
\Pi = \pi_{\sigma} \oplus \pi_{\sigma} \oplus \pi \oplus \pi_{\sigma}  \oplus \pi \oplus \pi_{\sigma}, \;\;H = H_A \otimes
H_B \otimes \C^6,
\end{eqnarray*}
\begin{eqnarray*}
\Dirac = \begin{bmatrix} \Dirac_A \otimes 1 & 1 \otimes \Dirac_B & 0 & 0  &0 &0 \\[1ex] 
1 \otimes \Dirac_B & -\Dirac_A \otimes 1 & 0 & 0 &0&0 \\[1ex] 
0 & 0 & 0 & 0 & -i \otimes \Dirac_B & \Dirac_A \otimes 1  \\[1ex] 
0 & 0 & 0 & 0 & \Dirac_A \otimes 1 & -i \otimes \Dirac_B \\[1ex] 
0 & 0 & i \otimes \Dirac_B & \Dirac_A \otimes 1  & 0 & 0 \\[1ex] 
0 & 0 &  \Dirac_A \otimes 1 & i \otimes \Dirac_B & 0 & 0 \\[1ex] 
\end{bmatrix}.
\end{eqnarray*}
If $\Dirac_A$ is invertible and the spectral triples $(\Alg,\Hil_A,\Dirac_A)$ and $(\Balg,\Hil_B,\Dirac_B)$
satisfy Rieffel's metric condition, so does the spectral triple on $E$.
\end{proposition}
In this case of split extensions other constructions are possible. For instance we can use the following representation and Dirac operator
\begin{eqnarray*}
\Pi = \pi_{\sigma} \oplus \pi_{\sigma} \oplus \pi \oplus \pi_{\sigma}, \;\;H = H_A \otimes
H_B \otimes \C^4,
\end{eqnarray*}
\begin{eqnarray*}
\Dirac = \begin{bmatrix} \Dirac_A \otimes 1 & 1
\otimes \Dirac_B & 0 & 0 \\[1ex] 1 \otimes \Dirac_B & -\Dirac_A \otimes 1 & 0 & 0 \\[1ex] 0 & 0 & 1 \otimes \Dirac_B & \Dirac_A \otimes 1  \\[1ex]  0 & 0 & \Dirac_A \otimes 1 & - 1 \otimes \Dirac_B \end{bmatrix},
\end{eqnarray*}
which seems more natural from the point of view of  K-homology.

\subsection{Extensions by compacts.} An extension by compacts is a short exact sequence of the form,
\begin{eqnarray}
\label{ExtCompact} \xymatrix{ 0 \ar[r] & \mathcal{K} \ar[r]^{\iota} & E \ar[r]^{\sigma} & A \ar[r] & 0}.
\end{eqnarray}
which we have mentioned before. From our point of view, these extensions correspond to the instance $B = \C$, the continuous functions on a single point. The canonical spectral triple on this space is the 'one-point' triple $(\C,\C,0)$. A second re-statement of Theorem \ref{ConstExt} is as follows:
\begin{proposition}
Let $A$ be a unital C$^*$-algebra, endowed with a spectral triple $(\Alg,H_A,\Dirac_A)$. Let $E \cong PAP + \mathcal{K}(PH_A)$ be a unital extension of $A$ by compact operators such that $[P,a]$ is a compact operator for each $a \in A$, $PAP \cap \mathcal{K}(P H_A) = \{0\}$ and the quadruple $(\Alg,H_A,\Dirac_A, P)$ is of Toeplitz type. Then $(\mathcal{E}, H, \Dirac)$, represented via $\Pi$, defines a spectral triple on $E$. Here,
\begin{eqnarray*}
\Pi = \pi_{\sigma} \oplus \pi_{\sigma} \oplus \pi \oplus \pi_{\sigma}  \oplus \pi \oplus \pi_{\sigma}, \;\;H = H_A  \otimes \C^6,
\end{eqnarray*}
\begin{eqnarray*}
\Dirac = \begin{bmatrix} \Dirac_A  & 0 & 0 & 0  & 0 &0 \\[1ex] 
0 & -\Dirac_A  & 0 & 0 &0&0 \\[1ex] 
0 & 0 & 0 & 0 &  \Dirac_A^q & \Dirac_A^p   \\[1ex] 
0 & 0 & 0 & 0 & \Dirac_A^p & \Dirac_A^q \\[1ex] 
0 & 0 & \Dirac_A^q & \Dirac_A^p   & 0 & 0 \\[1ex] 
0 & 0 &  \Dirac_A^p &  \Dirac_A^q & 0 & 0 \\[1ex] 
\end{bmatrix}.
\end{eqnarray*}
The spectral dimension of this triple is the same as the spectral dimension of $(\Alg, H_A, \Dirac_A)$. Moreover, if $\Dirac_A^p$ is invertible and the spectral triple $(\Alg,\Hil_A,\Dirac_A)$ satisfies the Rieffel metric condition then so does the spectral triple on $E$.
\end{proposition}
It is worth comparing our spectral triples with those considered by Christensen and Ivan \cite{CI1}. They make the same assumptions that we do, but the difference is that their triple acts on the Hilbert space $P H_A \oplus P H_A \oplus Q H_A$, rather than the  larger Hilbert space $H_A \otimes \C^6$. Their Dirac operator, like ours, is designed to obtain a spectral triple with good metric properties. In the spirit of Rieffel-Gromov-Hausdorff theory, Christensen-Ivan
introduce extra parameters $\alpha, \beta \in (0,1)$ which can be used to study the effects of "recovering" metric data on either the quotient algebra or the compacts itself, coming from the extension.

\subsection{Noncommutative spheres.}
\begin{example}\label{SU2} The quantum group $SU_q(2)$ was introduced by Woronowicz as a 1-parameter deformation of the
ordinary $SU(2)$ group \cite{Wor}. When one considers
the isomorphism $SU(2) \cong S^3$ of topological Lie groups, we can identify its C$^*$-algebra with a 1-parameter deformation of the continuous functions on the 3-sphere, $C(S_q^3)$,  for each $q \in [0,1]$. It can be formally defined as the universal C$^*$-algebra for generators $\alpha$ and $\beta$ subject to the
relations 
\begin{eqnarray*}
 \alpha^* \alpha + \beta^*\beta = I, \;\;\;\; \alpha \alpha^* + q^2\beta \beta^* = I,
\end{eqnarray*} 
\begin{eqnarray*} 
\alpha \beta = q \beta \alpha, \;\;\;\; \alpha \beta^* = q \beta^*\alpha,
\;\;\;\;\beta^*\beta = \beta \beta^*. 
\end{eqnarray*}
Woronowicz shows that the C$^*$-algebras $C(S_q^3)$ are all isomorphic for $q \in [0,1)$. For $q \in (0,1)$,
there is an alternative description of $C(S_q^3)$ as a symplectic foliation (see  \cite{CP1}, \cite{CPP3}, \cite{CPP4}): write $H := \ell_2(\N_0) \otimes
\ell_2(\Z)$ and let $S$ and $T$ be respectively the unilateral shift on $\ell_2(\N_0)$ and the bilateral shift on  $\ell_2(\Z)$, i.e $S e_k := e_{k+1}$ for each $k \geq 0$ and $T e_k := e_{k+1}$ for each $k \in \Z$. Let $N_q \in \mathcal{K}(\ell_2(\N_0))$ be defined by $N_q e_k := q^k e_k$. There exists a representation of $C(S_q^3)$ over $H$ defined by:
\begin{eqnarray*} 
\pi(\alpha) := S^* \sqrt{1 - N_q^2} \otimes I, \;\;\; \pi(\beta) &:=& N_q \otimes T^*.
\end{eqnarray*} 
and this representation is faithful. By considering the map $\sigma: C(S_q^3) \to C(\T)$ sending $\beta$ to $0$ and $\alpha$ to the generator $T^*$ of $C(\T)$, we soon obtain a short exact sequence,
\begin{eqnarray*} 
0 \;\rightarrow \; \mathcal{K} \otimes C(\T) \;\rightarrow \; C(S_q^3) \;\rightarrow \; C(\T)
\;\rightarrow \; 0. 
\end{eqnarray*} 
We obtain an isomorphism, 
\begin{eqnarray*}
C(S_q^3) \cong P C(\T) P \otimes \C I + \mathcal{K}(\ell_2(\N_0)) \otimes C(\T),
\end{eqnarray*}
where $P \in B(\ell_2(\Z))$ is the usual Toeplitz projection, with the property that $[P,x]$ is a
compact operator for each $x \in C(\T)$ and $PxP \otimes I \in C(S_q^3)$ for each $x \in C(\T)$. Note that we can write 
$$
\pi(\alpha) = -P T^* P(1 - \sqrt{1 - N_q^2}) \otimes I + P T^* P \otimes  I \textup{ whilst } \pi(\beta) \in \mathcal{K} \otimes C(\T).
$$ 
Because the algebra $C(S_q^3)$ has the requisite Toeplitz form, the construction in Theorem \ref{ConstExt} defines a spectral triple on $C(S_q^3)$ and it further provides $C(S_q^3)$ with the structure of a spectral metric space. For the latter, a slight perturbation of one of the Dirac operators is needed in this construction to ensure $P$-injectivity. In what follows, $\pi$ denotes the natural non-unital inclusion of $C(S_q^3)$ in $B(\ell_2(\Z) \otimes \ell_2(\Z))$, whilst $\pi_{\sigma}: C(S_q^3) \to B(\ell_2(\Z) \otimes \ell_2(\Z))$ is the map defined on the generators by $\pi_{\sigma}(\alpha) := T^* \otimes 1$, $\pi_{\sigma}(\beta) = 0$.

\begin{theorem}
Let $(\Alg, \ell_2(\Z), M_{\ell})$, $M_{\ell} e_n = n e_n$ be the usual spectral triple on $C^1(\T)$, where $\Alg \subseteq C(\T)$ is any dense $^*$-subalgebra of $C(\T)$ such that $(\Alg, \ell_2(\Z), \Dirac)$ is a triple satisfying 
\begin{eqnarray*}
[M_{\ell}, f] \in B(\ell_2(\Z)), \;\;[|M_{\ell}|, f] \in B(\ell_2(\Z)), \;\;f \in \Alg
\end{eqnarray*}
(e.g. $\Alg = C^1(\T)$).
Then, for each $\lambda \in \R$ and for each $q \in (0,1)$, there is a spectral triple $(\mathcal{E}, (\ell_2(\Z) \otimes \ell_2(\Z)) \otimes \C^6, \Dirac_{\lambda})$ on $C(S_q^3)$, represented via $\pi_{\sigma} \oplus \pi_{\sigma} \oplus \pi \oplus \pi_{\sigma} \oplus \pi \oplus \pi_{\sigma}$
and where 
\begin{eqnarray*}
\Dirac_{\lambda} = 
 \begin{bmatrix} M_{\ell,\lambda} \otimes 1 & 1
\otimes M_{\ell} & 0 \\[1ex] 
1 \otimes M_{\ell} & - M_{\ell,\lambda} \otimes 1 & 0\\[1ex] 
0 & 0  & \Dirac_I \\[1ex]
\end{bmatrix},
\end{eqnarray*}
and
\begin{eqnarray*}
\Dirac_I=
\begin{bmatrix}
0 & 0 & M_{\ell,\lambda}^q \otimes 1 - i \otimes M_{\ell} & M_{\ell,\lambda}^p \otimes 1  \\[1ex] 
0 & 0 & M_{\ell,\lambda}^p \otimes 1  & - M_{\ell,\lambda}^q \otimes 1  - i \otimes M_{\ell} \\[1ex] 
M_{\ell,\lambda}^q \otimes 1 + i \otimes M_{\ell} & M_{\ell,\lambda}^p \otimes 1 &  0 & 0  \\[1ex] 
M_{\ell,\lambda}^p \otimes 1  & - M_{\ell,\lambda}^q \otimes 1  + i \otimes M_{\ell} &  0 & 0 \\[1ex] 
\end{bmatrix}.
\end{eqnarray*}
(Here, $M_{\ell,\lambda} := (M_{\ell} + \lambda I)$.) This spectral triple has spectral dimension 2. Moreover, for each $\lambda > 0$, the spectral triple implements the structure of a quantum metric space on $C(S_q^3)$.
\end{theorem}

There are numerous other constructions of spectral triples on the algebra $C(SU_q(2))$ in the literature, mostly with different spectral dimensions and no information about Rieffel's metric condition. The precise relation between those and our construction is unclear. The first spectral triples on  $C(SU_q(2))$ were constructed by Chakraborty and Pal in \cite{CPP1} and \cite{CPP2}, whose focus was very different to ours. The named authors show that any spectral triple on $C(SU_q(2))$ which is of a certain natural form and which is equivariant for the quantum group co-action of $SU_q(2)$ must have spectral dimension at least 3, which is in contrast to our spectral triple of dimension 2. In \cite{CPP3} the same authors construct spectral triples on $C(SU_q(2))$ using an altogether different approach, focusing on those triples which are equivariant for the action of $\T^2$ on $C(SU_q(2))$, which might be closer to our spectral triple. The construction in \cite{CPP2} was used and further developed by Connes \cite{ConnesSU}. A different construction of a $3^+$-summable spectral triple on   $C(SU_q(2))$ was developped in \cite{DiracSU} using the classical Dirac operator. In another paper \cite{localSU} the same authors give a construction of this triple via an extension using the cosphere bundle defined in \cite{ConnesSU} which appears somewhat similar to our construction.

\end{example}

\begin{example}\label{S2q}
The Podle\'s spheres were introduced as a family of \textit{quantum homogeneous spaces} for the action of the
quantum $SU(2)$ group \cite{Pod}. Probably the most widely studied algebraically
non-trivial examples are the so-called
\textit{equatorial Podle\'s spheres}. They can be defined for each $q \in (0,1)$ as the universal C$^*$-algebra,
$C(S_q^2)$, for generators $\alpha$ and $\beta$,  subject to the relations,
\begin{eqnarray*}
\beta^{^*} = \beta, \;\;\beta \alpha  = q \alpha \beta, \;\;\alpha^* \alpha + \beta^2 = I, \;\; q^4 \alpha
\alpha^* + \beta^2 = q^4.
\end{eqnarray*}
Using the same notation as in Example \ref{SU2}, we can write down a representation of $C(S_q^2)$ over $H := \ell_2(\N) \otimes \C^2$ defined by:
\begin{eqnarray*} 
\pi(\alpha) := T \sqrt{1 - N_q^4} \otimes \begin{bmatrix} 1 & 0 \\ 0 & 1 \end{bmatrix}, \;\;\;
\pi(\beta) &:=& N_q^2 \otimes \begin{bmatrix} 1 & 0
\\ 0 & -1 \end{bmatrix}, 
\end{eqnarray*} 
and this representation is faithful. By considering the map $\sigma: C(S_q^2) \to C(\T)$ sending $\beta$ to $0$ and $\alpha$ to $T \in C(\T)$, we soon obtain a short exact sequence,
\begin{eqnarray*}
0 \;\rightarrow \; \mathcal{K} \otimes \C^2 \;\rightarrow \; C(S_q^2) \;\rightarrow \; C(\T)
\;\rightarrow \; 0.
\end{eqnarray*}
We obtain an isomorphism,
\begin{eqnarray*}
C(S_q^2) \cong P C(\T) P \otimes \C I + \mathcal{K}(\ell_2(\N_0)) \otimes \C^2,
\end{eqnarray*}
where $P \in B(\ell_2(\Z))$ is again the usual Toeplitz projection, so that again $[P,x]$ is 
compact for each $x \in \C(\T)$ and now $PxP \otimes 1 \in C(S_q^2)$ for each $x \in C(\T)$. As before we can write 
$$
\pi(\alpha) = -P T P(1 - \sqrt{1 - N_q^4}) \otimes I + P T P \otimes  I, \textup{ and }   \pi(\beta) \in \mathcal{K} \otimes \C^2.
$$ 
As in Example \ref{SU2}, we can formulate the existence of spectral triples for the algebras $C(S_q^2)$ as follows: first, on $B$ we introduce the two-point triple, which turns $B$ into a spectral metric space
\begin{eqnarray*}
\left(\C^2, \C^2, \gamma := \begin{bmatrix} 0 & 1 \\ 1 & 0 \end{bmatrix} \right).
\end{eqnarray*}
Let $\pi$ denote the natural non-unital inclusion of $C(S_q^2)$ in $B(\ell_2(\Z) \otimes \C^2)$, whilst $\pi_{\sigma}: C(S_q^2) \to B(\ell_2(\Z) \otimes \C^2)$ is the map defined on the generators by $\pi_{\sigma}(\alpha) := T \otimes I_2$, $\pi_{\sigma}(\beta) = 0$.
\begin{theorem}\label{Podles}
Let $(\Alg, \ell_2(\Z), M_{\ell})$, $M_{\ell} e_n = n e_n$ be the usual spectral triple on $C^1(\T)$, where $\Alg \subseteq C(\T)$ is any dense $^*$-subalgebra of $C(\T)$ such that $(\Alg, \ell_2(\Z), \Dirac)$ satisfies 
\begin{eqnarray*}
[M_{\ell}, f] \in B(\ell_2(\Z)), \;\;[|M_{\ell}|, f] \in B(\ell_2(\Z)), \;\;f \in \Alg
\end{eqnarray*}
(e.g. $\Alg = C^1(\T)$).
Then, for each $\lambda \in \R$ and for each $q \in (0,1)$, there is a spectral triple $(\mathcal{E}, (\ell_2(\Z) \otimes \C^2) \otimes \C^6, \Dirac_{\lambda})$ on $C(S_q^2)$, represented via $\pi_{\sigma} \oplus \pi_{\sigma} \oplus \pi \oplus \pi_{\sigma}$
and where
\begin{eqnarray*}
\Dirac_{\lambda} = 
 \begin{bmatrix} 
M_{\ell,\lambda} \otimes 1 & 1 \otimes \gamma & 0 \\[1ex] 
1 \otimes \gamma & - M_{\ell,\lambda} \otimes 1 & 0\\[1ex] 
0 & 0 & \Dirac_I \\[1ex]
\end{bmatrix},
\end{eqnarray*}
and
\begin{eqnarray*}
\Dirac_I=
\begin{bmatrix}
0 & 0 & M_{\ell,\lambda}^q \otimes 1 - i \otimes \gamma & M_{\ell,\lambda}^p \otimes 1  \\[1ex] 
0 & 0 & M_{\ell,\lambda}^p \otimes 1  & - M_{\ell,\lambda}^q \otimes 1  - i \otimes \gamma \\[1ex] 
M_{\ell,\lambda}^q \otimes 1 + i \otimes \gamma & M_{\ell,\lambda}^p \otimes 1 &  0 & 0  \\[1ex] 
M_{\ell,\lambda}^p \otimes 1  & - M_{\ell,\lambda}^q \otimes 1  + i \otimes \gamma &  0 & 0 \\[1ex] 
\end{bmatrix}.
\end{eqnarray*}
(Here, $M_{\ell,\lambda} := (M_{\ell} + \lambda I)$.) This spectral triple has spectral dimension 1. Moreover, for each $\lambda > 0$, the spectral triple implements the structure of a quantum metric space on $C(S_q^2)$.
\end{theorem}  
A spectral triple on  $C(S_q^2)$ of dimension 2 has been constructed previously in \cite{DDLW}, again with no information about the metric condition. Also here the connection to our construction is unclear and left to future research. The relation seems even less clear than in the previous Example \ref{SU2} since the construction in \cite{DDLW} does not use any extensions.
\end{example}

The noncommutative $n$-spheres for higher dimensions can be defined inductively on $n$. The spheres of odd
dimension arise as short exact sequences of the form
\begin{eqnarray*} 
0 \rightarrow \mathcal{K} \otimes C(S^1) \rightarrow C(S_q^{2n + 1}) \rightarrow C(S_q^{2n - 1})
\rightarrow 0, \;\;n \geq 1. 
\end{eqnarray*} 
and the spheres of even dimension as short exact sequences of the form 
\begin{eqnarray*} 
0 \rightarrow \mathcal{K}
\otimes \C^2 \rightarrow C(S_q^{2n}) \rightarrow
C(S_q^{2n - 1}) \rightarrow 0, \;\;n \geq 1. 
\end{eqnarray*}
We suspect that the same process that was used to construct spectral metric spaces on $C(S_q^2)$ and $C(S_q^3)$ can, via this procedure, lead to the construction of spectral metric spaces for 1-parameter quantum spheres of any integer dimension. We can then relate these to similar constructions in the literature, e.g. \cite{CPP3}, \cite{CPP4}.

\section{Outlook.}\label{out}

In addition to the questions mentioned at the end of  Example \ref{SU2} and \ref{S2q}
we briefly raise a number of questions related to this article which seem interesting. 

\medskip

\noindent
1. The spectral triple we construct on the extension (\ref{extension}) behaves well with respect to summability and induces metrics on the state space. This was our main goal. However, the following question still remains:

\begin{question}
What is the KK-theoretical meaning of the spectral triple  we construct on the extension?
\end{question}

\medskip

\noindent
2. Our construction of spectral triples is restricted to a special class of extensions (Toeplitz type extensions) but is applicable to several concrete examples as demonstrated in the last section. As discussed, there are similarities between a general semisplit extension by a stable ideal the Toeplitz type extensions we consider.  

\begin{question}
Can the construction of the spectral triple in Thm.\ref{ConstExt} be generalised to extensions which are not necessarily of Toeplitz type or only of Toeplitz type in a generalised sense?
\end{question}

\medskip

\noindent
3. Rieffel proposed a notion of distance between compact quantum metric spaces, modelled on the Gromov-Hausdorff
distance (\cite{Ri1, Ri2, Ri3}). It has since been used in a
number of questions relating to C$^*$-algebras endowed with seminorms. Some of the results are quite surprising:
Rieffel (\cite{RI4}) shows how the common observation in  quantum physics that `matrices converge to the 2-sphere' can be
illustrated quite well using Rieffel-Gromov-Hausdorff
convergence.

There are various perspectives that we could take with respect to convergence for extensions in this chapter, especially for algebras arising as $q$-deformations.
One is to try to mimic the convergence studied by Christensen and Ivan in their approach. They construct
a two-parameter family of spectral
triples $(\mathcal{T}, H_A, \Dirac_{\alpha,\beta})$ for extensions of the form 
\begin{eqnarray*} 0 \rightarrow
\mathcal{K} \rightarrow \mathcal{T} \rightarrow A
\rightarrow 0, \end{eqnarray*} and for $\alpha, \beta > 0$, for which the quantum metric spaces converge to those on $A$ and $\mathcal{K}$ as $\alpha \to 0$ and $\beta \to 0$. However,  for example in the case of the Podle\'s spheres this turns out  not to be sufficient to study the Gromov-Hausdorff convergence aspects of varying the parameter $q$. The following two questions seem interesting, though we point out that the situation addressed in those questions is quite different from the matrix algebra convergence in \cite{RI4} since the parameter $q$ does not change the algebras of the Podles spheres. We remark that classically $q$ can be regarded as a label for Poisson structures on $S^2$ (\cite{Roy}).

\begin{question}
Suppose that $(q_n)_{n \in \N} \subseteq (0,1)$ is a sequence converging to $q \in (0,1)$ and let
$(\Alg(C(S_{q_n}^2)),L)$ be one of the compact quantum metrics on
the Podle\'s sphere $C(S_{q_n}^2)$ for $n \in \N$ as defined in Thm.\ref{Podles}. Is it true that
$(\Alg(C(S_{q_n}^2)),L)$ converges to $(\Alg(C(S_{q}^2)),L)$ for
Rieffel-Gromov-Hausdorff convergence? \end{question}

\begin{question}
Suppose now that $(q_n)_{n \in \N} \subseteq
(0,1)$ converges to $1$. Let $(C^1(S^2),
L_{\Dirac})$ be the usual Lipschitz seminorm on the algebra $C(S^2) \cong C(S_1^2)$ for which the restriction of the metric to $S^2$ is the geodesic metric. Is it true that $(\Alg(C(S_{q_n}^2)),L)$ converges to $(C^1(S^2)),L)$, or any equivalent Lipschitz pair on the
two-sphere, for Rieffel-Gromov-Hausdorff convergence?

\end{question}

\bibliographystyle{plain}

\bibliography{biblio}

\begin{thebibliography}{10}

\bibitem{BJ}
Saad Baaj and Pierre Julg.
\newblock Th\'eorie bivariante de {K}asparov et op\'erateurs non born\'es dans
  les {$C^{\ast} $}-modules hilbertiens.
\newblock {\em C. R. Acad. Sci. Paris S\'er. I Math.}, 296(21):875--878, 1983.

\bibitem{BMR}
Jean~V. Bellissard, Matilde Marcolli, and Kamran Reihani.
\newblock Dynamical systems on spectral metric spaces.
\newblock {\em arXiv:math.OA/1008.4617}, 2010.

\bibitem{Bla}
Bruce Blackadar.
\newblock {\em {$K$}-theory for operator algebras}, volume~5 of {\em
  Mathematical Sciences Research Institute Publications}.
\newblock Cambridge University Press, Cambridge, second edition, 1998.

\bibitem{CP1}
Partha Chakraborty.
\newblock From {$C^\ast$}-algebra extensions to compact quantum metric spaces,
  quantum {SU}(2), {P}odle\'s spheres and other examples.
\newblock {\em J. Aust. Math. Soc.}, 90(1):1--8, 2011.

\bibitem{CPP1}
Partha Chakraborty and Arupkumar Pal.
\newblock Equivariant spectral triples on the quantum {${\rm SU}(2)$} group.
\newblock {\em $K$-Theory}, 28(2):107--126, 2003.

\bibitem{CPP2}
Partha Chakraborty and Arupkumar Pal.
\newblock On equivariant {D}irac operators for {${\rm SU}_q(2)$}.
\newblock {\em Proc. Indian Acad. Sci. Math. Sci.}, 116(4):531--541, 2006.

\bibitem{CPP3}
Partha Chakraborty and Arupkumar Pal.
\newblock Torus equivariant spectral triples for odd-dimensional quantum
  spheres coming from {$C^*$}-extensions.
\newblock {\em Lett. Math. Phys.}, 80(1):57--68, 2007.

\bibitem{CPP4}
Partha Chakraborty and Arupkumar Pal.
\newblock Characterization of {${\rm SU}_q(\l+1)$}-equivariant spectral triples
  for the odd dimensional quantum spheres.
\newblock {\em J. Reine Angew. Math.}, 623:25--42, 2008.

\bibitem{Chr}
Erik {Christensen}.
\newblock {On weakly D-differentiable operators}.
\newblock {\em arXiv:math.OA/1303.7426}, 2013.

\bibitem{CI2}
Erik Christensen and Cristina Ivan.
\newblock Spectral triples for {AF} {$C^*$}-algebras and metrics on the
  {C}antor set.
\newblock {\em J. Operator Theory}, 56(1):17--46, 2006.

\bibitem{CI1}
Erik Christensen and Cristina Ivan.
\newblock Extensions and degenerations of spectral triples.
\newblock {\em Comm. Math. Phys.}, 285(3):925--955, 2009.

\bibitem{Con2}
Alain Connes.
\newblock Compact metric spaces, {F}redholm modules, and hyperfiniteness.
\newblock {\em Ergodic Theory Dynam. Systems}, 9(2):207--220, 1989.

\bibitem{Con1}
Alain Connes.
\newblock {\em Noncommutative geometry}.
\newblock Academic Press Inc., San Diego, CA, 1994.

\bibitem{ConnesSU}
Alain Connes.
\newblock Cyclic cohomology, quantum group symmetries and the local index
  formula for {$SU_q(2)$}.
\newblock {\em J. Inst. Math. Jussieu}, 3:17--68, 2004.

\bibitem{Con3}
Alain Connes.
\newblock On the spectral characterization of manifolds.
\newblock {\em J. Noncommut. Geom.}, 7(1):1--82, 2013.

\bibitem{CMh2}
Alain Connes and Henri Moscovici.
\newblock The local index formula in noncommutative geometry.
\newblock {\em Geom. Funct. Anal.}, 5(2):174--243, 1995.

\bibitem{DDLW}
Ludwik Dabrowski, Francesco D'Andrea, Giovanni Landi, and Elmar Wagner.
\newblock Dirac operators on all {P}odle\'s quantum spheres.
\newblock {\em J. Noncommut. Geom.}, 1(2):213--239, 2007.

\bibitem{DiracSU}
Ludwik Dabrowski, Giovanni Landi, Andrzej Sitarz, Walter van Suijlekom, and
  Joseph~C. Varilly.
\newblock The {D}irac operator on {$SU_q(2)$}.
\newblock {\em Comm. Math. Phys.}, 259:729--759, 2005.

\bibitem{localSU}
Ludwik Dabrowski, Giovanni Landi, Andrzej Sitarz, Walter van Suijlekom, and
  Joseph~C. Varilly.
\newblock The local index formula for {$SU_q(2)$}.
\newblock {\em K-Theory}, 35:375--394, 2005.

\bibitem{GG}
Olivier {Gabriel} and Martin {Grensing}.
\newblock {Spectral triples and generalized crossed products}.
\newblock {\em arXiv:math.OA/1310.5993}, 2013.

\bibitem{GM}
Markus {Goffeng} and Bram {Mesland}.
\newblock {Spectral triples and finite summability on Cuntz-Krieger algebras}.
\newblock {\em Doc. Math}, 20:89--170 (electronic), 2015.

\bibitem{HSWZ}
Andrew Hawkins, Adam Skalski, Stuart White, and Joachim Zacharias.
\newblock On spectral triples on crossed products arising from equicontinuous
  actions.
\newblock {\em Math. Scand.}, 113(2):262--291, 2013.

\bibitem{HR}
Nigel Higson and John Roe.
\newblock {\em Analytic {$K$}-homology}.
\newblock Oxford Mathematical Monographs. Oxford University Press, Oxford,
  2000.

\bibitem{Kas3}
Gennady Kasparov.
\newblock The operator {$K$}-functor and extensions of {$C^{\ast} $}-algebras.
\newblock {\em Izv. Akad. Nauk SSSR Ser. Mat.}, 44(3):571--636, 719, 1980.

\bibitem{Lat1}
Fr{\'e}d{\'e}ric Latr{\'e}moli{\`e}re.
\newblock Bounded-{L}ipschitz distances on the state space of a
  {$C^*$}-algebra.
\newblock {\em Taiwanese J. Math.}, 11(2):447--469, 2007.

\bibitem{Lat2}
Fr{\'e}d{\'e}ric Latr{\'e}moli{\`e}re.
\newblock Quantum locally compact metric spaces.
\newblock {\em J. Funct. Anal.}, 264(1):362--402, 2013.

\bibitem{LRV}
Steven Lord, Adam Rennie, and Joseph~C. V{\'a}rilly.
\newblock Riemannian manifolds in noncommutative geometry.
\newblock {\em J. Geom. Phys.}, 62(7):1611--1638, 2012.

\bibitem{NT}
Sergey Neshveyev and Lars Tuset.
\newblock The {D}irac operator on compact quantum groups.
\newblock {\em J. Reine Angew. Math.}, 641:1--20, 2010.

\bibitem{OR}
Narutaka Ozawa and Marc~A. Rieffel.
\newblock Hyperbolic group {$C^*$}-algebras and free-product {$C^*$}-algebras
  as compact quantum metric spaces.
\newblock {\em Canad. J. Math.}, 57(5):1056--1079, 2005.

\bibitem{Pod}
Piotr Podle{\'s}.
\newblock Quantum spheres.
\newblock {\em Lett. Math. Phys.}, 14(3):193--202, 1987.

\bibitem{Ren1}
Adam Rennie.
\newblock Summability for nonunital spectral triples.
\newblock {\em $K$-Theory}, 31(1):71--100, 2004.

\bibitem{Ren}
Adam Rennie.
\newblock {Spectral triples: examples and applications, notes for lectures
  given at the workshop on noncommutative geometry and physics, Yokohama},
  2009.

\bibitem{RV}
Adam Rennie and Josef~C. Varilly.
\newblock Reconstruction of manifolds in noncommutative geometry.
\newblock {\em arXiv:math.OA/0610418}, 2006.

\bibitem{Ri1}
Marc~A. Rieffel.
\newblock Metrics on states from actions of compact groups.
\newblock {\em Doc. Math.}, 3:215--229 (electronic), 1998.

\bibitem{Ri2}
Marc~A. Rieffel.
\newblock Metrics on state spaces.
\newblock {\em Doc. Math.}, 4:559--600 (electronic), 1999.

\bibitem{Ri3}
Marc~A. Rieffel.
\newblock Compact quantum metric spaces.
\newblock In {\em Operator algebras, quantization, and noncommutative
  geometry}, volume 365 of {\em Contemp. Math.}, pages 315--330. Amer. Math.
  Soc., Providence, RI, 2004.

\bibitem{RI4}
Marc~A. Rieffel.
\newblock {\em Gromov-{H}ausdorff distance for quantum metric spaces. {M}atrix
  algebras converge to the sphere for quantum {G}romov-{H}ausdorff distance}.
\newblock American Mathematical Society, Providence, RI, 2004.
\newblock Mem. Amer. Math. Soc. {{\bf{1}}68} (2004), no. 796.

\bibitem{Roy}
Dmitry Roytenberg.
\newblock Poisson cohomology of {$SU(2)$}-covariant "necklace'' poisson
  structures on {$S^2$}.
\newblock {\em J. Nonlinear Math. Phys.}, 9(3):347--356, 2002.

\bibitem{Var}
Joseph~C V{\'a}rilly and Pawel Witkowski.
\newblock Dirac operators and spectral geometry, 2006.

\bibitem{Wan}
Xiaolu Wang.
\newblock Voiculescu theorem, {S}obolev lemma, and extensions of smooth
  algebras.
\newblock {\em Bull. Amer. Math. Soc. (N.S.)}, 27(2):292--297, 1992.

\bibitem{Wor}
Stanislaw~L. Woronowicz.
\newblock Twisted {${\rm SU}(2)$} group. {A}n example of a noncommutative
  differential calculus.
\newblock {\em Publ. Res. Inst. Math. Sci.}, 23(1):117--181, 1987.

\end{thebibliography}

\end{document}